\documentclass[12pt]{article}
\providecommand{\keywords}[1]{\textbf{Keywords:} #1}
\usepackage{graphicx} 
\usepackage{authblk}
\usepackage{etoolbox}
\usepackage{geometry}
\usepackage{amsthm,amsmath,amssymb}
\usepackage[utf8]{inputenc}
\usepackage{hyperref} 
\usepackage{cleveref}
\usepackage{makecell}
 \usepackage[style=numeric, backend=biber, sorting=none]{biblatex}
\renewbibmacro*{in:}{}
\usepackage{algorithmic}
\usepackage{algorithm}

\usepackage[markup=underlined]{changes}
\definechangesauthor[name={Reviewer 1}, color=orange]{R2}
\usepackage{verbatim}

\crefrangelabelformat{equation}{(#3#1#4--#5#2#6)}
\crefrangeformat{equation}{#3(#1#4)--(#2#6)} 
\usepackage{enumitem}
\newlist{Case}{enumerate}{2}
\setlist[Case, 1]{%
    label           =   {\bfseries Case \arabic*.},
    labelindent=1em ,labelwidth=1.3cm, labelsep*=1em, leftmargin =!
}
\setlist[Case, 2]{%
    label           =   {\bfseries Subcase \arabic{Casei}.\arabic*.},
    labelindent=-1em ,labelwidth=1.3cm, labelsep*=1em, leftmargin =!
}

\usepackage{xcolor}    
\usepackage{subcaption}
\usepackage{xspace}
\definecolor{omittext}{RGB}{255,0,0}  

\usepackage{booktabs}
\usepackage{graphicx}
\usepackage{multirow}
\usepackage{float}
\usepackage{caption}
\usepackage{soul, color, xcolor}
\usepackage{abstract}

\newcommand{\St}{\mathrm{St}}
\def\Hess{\mathrm{Hess}}
\def\grad{\mathrm{grad}}
\def\dist{\mathrm{dist}}
\def\Proj{\mathrm{Proj}}
\def\trace{\mathrm{trace}}
\def\D{\mathrm{D}}
\def\St{\mathrm{St}}
\def\T{\mathrm{T}}
\def\F{\mathrm{F}}
\def\rank{\mathrm{rank}}
\def\P{\mathrm{P}}
\newtheorem{theorem}{Theorem}[section]
\newtheorem{proposition}[theorem]{Proposition}
\newtheorem{Lemma}[theorem]{Lemma}
\newtheorem{Definition}[theorem]{Definition}
\newtheorem{Assumption}[theorem]{Assumption}

\newtheorem{remark}{Remark}

\usepackage{xcolor}
\usepackage{xpatch}
\makeatletter
\ExplSyntaxOn
\cs_new:Npn \bibColoredItems #1#2
  {
    \clist_map_inline:nn {#2} { \cs_new:cpn {bib@colored@##1} {#1} } 
  }
\ExplSyntaxOff

\newcommand\bib@setcolor[1]{%
  \ifcsname bib@colored@#1\endcsname
    \expanded{\noexpand\color{\csname bib@colored@#1\endcsname}}%
  \else
    \normalcolor
  \fi
}

\IfPackageLoadedTF{hyperref}{\@tempswatrue}{\@tempswafalse}
\if@tempswa
  \xpatchcmd\@bibitem {\H@item}{\bib@setcolor{#1}\H@item}{}{\PatchFailed}
  \xpatchcmd\@lbibitem{\H@item}{\bib@setcolor{#2}\H@item}{}{\PatchFailed}
\else
  \xpatchcmd\@bibitem {\item}  {\bib@setcolor{#1}\item}  {}{\PatchFailed}
  \xpatchcmd\@lbibitem{\item}  {\bib@setcolor{#2}\item}  {}{\PatchFailed}
\fi
\makeatother

\definecolor{revisioncolor}{HTML}{3370BD} 

\title{An Inexact Riemannian Gradient Descent Algorithm on the Stiefel Manifold with One Newton-Schulz Iteration}

\author[1]{Yuqiu Su}
\affil[1]{School of Mathematical Sciences, Xiamen University, Xiamen, China}
\author[1,2]{Wen Huang}
\affil[2]{Corresponding author: \url{wen.huang@xmu.edu.cn}}

\begin{document}
 


\maketitle
\begin{abstract}
In this paper, we propose an inexact Riemannian gradient descent algorithm on the Stiefel manifold (IRGS-StieONS) using an adaptive step size, where the ``inexact'' refers to the inexactness of retraction. It is proven that one single Newton-Schulz iteration for the retraction is sufficient for global convergence and local linear convergence. Compared to the landing and augmented Lagrangian-based algorithms, the proposed algorithm is the first infeasible algorithm that permits adaptive step sizes with a practical initial step size and guarantees global convergence and local linear convergence under mild assumptions. Moreover, we show that the local convergence rate depends on the condition number of the Riemannian Hessian, which matches the Riemannian steepest descent algorithm. This result implies that the infeasibility in the proposed algorithm does not influence the local convergence rate. Furthermore, a stochastic gradient version of IRGD‑StieONS is proposed and is shown to achieve the same convergence rate as Riemannian stochastic gradient descent with decreasing step size. 
Numerical experiments demonstrate that both IRGD‑StieONS and its stochastic counterpart exhibit superior performance and robustness.
\end{abstract}
\keywords{Riemannian optimization,  Stiefel manifold, inexact first-order information, infeasible methods, Newton-Schulz iteration}

\section{Introduction}
In this paper, we consider the problem of optimizing $f$ with orthogonality constraints:  
\begin{equation}\label{main problem}
    \min_{X\in \St (p,n)}f(X),
\end{equation}
where $\mathrm{St}(p,n) = \{X\in \mathbb{R}^{n\times p}:X^TX = I_p\}$ is the compact Stiefel manifold 
and $f:\mathbb{R}^{n\times p}\to \mathbb{R}$ is continuously differentiable. Many important applications can be formulated as optimization problems in the form of~\eqref{main problem}, such as the orthogonal procrustes problem~\cite{LH1999procrustes}, principal
component analysis~\cite{HH1933PCA}, and deep learning~\cite{AM2016Unitary, Eugene2017recurrentnetworks}.


Due to the manifold structure of the orthogonality constraint, there have been many Riemannian optimization algorithms proposed for solving the above problems, including the Riemannian gradient descent  algorithm~\cite{Absil2008OAOMM}, the Riemannian nonlinear conjugate gradient algorithm~\cite{HT2013RCG}, the Riemannian trust region algorithm~\cite{ABG2007RTR}, the Riemannian BFGS methods~\cite{HuaAbsGal2018}, and the Riemannian truncated Newton's method~\cite{Huang2025NRCG}. A crucial part of Riemannian optimization methods is the use of retraction, which is a map that takes inputs $X\in\St(p,n)$ and $Z$, a tangent vector at $X$, and outputs $Y\in \St(p,n)$. It defines a way to move on the manifold $\St(p,n).$
Therefore, retraction-based algorithms are referred to as $\textit{feasible}$ methods on $\St(p, n)$. The commonly used retractions on the Stiefel manifold include the exponential map~\cite{RH2022Exponentia}, the polar-based retraction~\cite{Absil2008OAOMM, Rasmus2026polar}, the Cayley retraction~\cite{Wenzaiwen2013orth}, and the QR-based retraction~\cite{Absil2008OAOMM}. However, they all involve linear algebra operations
on matrices like inversion, square root, or exponential, which makes parallelization particularly difficult for optimization problems with orthogonality constraints~\cite{BinGao2019paraOrth}.

To address the aforementioned issues, many $\textit{infeasible}$ methods have been proposed. These methods do not require a retraction; consequently, the iterates are not necessarily confined to the manifold.
Based on the framework of the augmented Lagrangian method, Gao et al.~\cite{BinGao2019paraOrth} proposed a proximal linearized augmented Lagrangian (PLAM) algorithm method and its column-wise normalization version (PCAL) and
established global subsequence convergence, worst-case complexity, and local convergence rate for PLAM under some mild assumptions. This algorithm does not require a retraction and is parallelizable. The global convergence of PLAM requires a sufficiently large penalty parameter and, correspondingly, a small step size. 
Xiao et al.~\cite{xiao2022class, xiao2024solving} proposed two exact penalty functions, PenC and ExPen, to reformulate the
constrained problem~\eqref{main problem} into an unconstrained one by employing different penalty
functions, where the ``exact'' refers to that in a neighborhood of $\St(p, n)$, any first- and second-order stationary points of the penalty functions are also those of Problem~\eqref{main problem}, but the converse does not necessarily hold. PenC~\cite{xiao2022class}  requires high smoothness of the original objective function, and calculating its gradient is usually expensive in practice. ExPen~\cite{xiao2024solving} was built to overcome the above problems. A global convergence of a nonlinear conjugate gradient method for ExPen was established therein.


On the other hand, a prevailing infeasible method at present is the landing algorithm~\cite{Pierre2022Fastlanding, gao2022optimization, ablin2024infeasible, SYB2024Local,sun2024retraction}, which incorporates a relative gradient term and a penalty term to ensure that the trajectory ultimately lands on the manifold during optimization. The landing algorithm primarily involves matrix multiplication in the iterative process, so it can be parallelized. In~\cite{Pierre2022Fastlanding}, Ablin and Peyr{\' e} proposed and analyzed the landing method  in the case of
the orthogonal manifold, and
prove a sub-optimal $K^{-\frac{1}{3}}$ convergence rate with decreasing step-sizes. Gao et al.~\cite{gao2022optimization} extend the landing flow to solve the optimization problem (\ref{main problem}) over the Stiefel manifold and give the landing flow a geometric interpretation. Ablin et al.~\cite{ablin2024infeasible} constructed a new merit function to improve the theoretical result of~\cite{Pierre2022Fastlanding}. They show that
the basic landing method with constant step size achieves $K^{-1}$ convergence rate. Moreover, 
a landing stochastic gradient descent algorithm and a landing SAGA 
algorithm are proposed therein for the case where the objective function \(f\) is the average of \(N\) component functions and provided a convergence analysis consistent with that of the corresponding traditional Riemannian optimization algorithms employing retraction. Sun et al.~\cite{SYB2024Local, sun2024retraction} assume the local Riemannian Polyak-$\text{\L}$ojasiewicz (P$\text{\L}$) condition and incorporated a merit function into the analysis, thereby establishing the local linear convergence of the landing algorithm. They also extended this framework to distributed optimization. In all the above references, the convergence analysis of the landing algorithm is restricted to fixed~\cite{Pierre2022Fastlanding,SYB2024Local,ablin2024infeasible, sun2024retraction} or diminishing~\cite{ablin2024infeasible} step sizes. 


Recently, many variants~\cite{goyens2026riemannian,si2026unified, xiong2026second} of landing algorithms have been proposed. 
Goyens and Feppon~\cite{goyens2026riemannian} proposed a globally convergent landing method with a backtracking line search.
However, their analysis only guarantees global convergence and does not provide a local convergence rate. Beyond orthogonality constraints, Si et al.~\cite{si2026unified} generalized the landing framework to equality-constrained optimization and developed a backtracking-based landing method. The sublinear convergence rate for both feasibility and
stationarity is further obtained.  More recently, Xiong et al.~\cite{xiong2026second} proposed a second-order method landing on the Stiefel manifold, which is proved to enjoy local quadratic (or superlinear for its inexact variant) convergence under unit-step.

Another class of infeasible methods builds upon Riemannian optimization algorithms while employing inexact information. Specifically, these methods approximate computationally expensive quantities—such as the retraction—through iterative procedures. The numerical errors introduced by such approximations place them in the category of infeasible methods. For instance, Li et al.~\cite{Li2020Efficient} proposed an iterative estimation of the Cayley transform that efficiently uses only
matrix multiplications and extended the existing algorithm to propose Cayley stochastic gradient descent (SGD) and Cayley ADAM algorithms on the Stiefel manifold. They also analyzed the convergence of the iterative Cayley transform and the sublinear convergence rate of Cayley SGD. Li et al. \cite{YuchenLi2024IFROA} analyzed an inexact Riemannian gradient descent algorithm in a general framework of tangential Block Majorization-Minimization, where Riemannian gradient and retraction are inexactly computed. But under the assumption that one can choose $v_k\in \T_{x_k}\mathcal{M}$ such that
$$R_{x_k}(v_k) = \widehat{R}_{x_k}(\widehat{v}_k),$$
where $\widehat{R}_{x_k}, \widehat{v}_k$ denote inexact retraction and inexact Riemannian gradient respectively. Therefore, the inexact retraction step can be viewed as an exact retraction $R$ on $v_k$. Recently, Peng et al.~\cite{peng2026ns} proposed a Newton-Schulz-based Riemannian gradient method on the orthogonal group synchronization, in which the exact SVD-based retraction is replaced by an inexact retraction obtained from a few Newton-Schulz iterations. Moreover, they established a high-probability linear convergence guarantee for their Newton-Schulz-based inexact Riemannian gradient method under a suitable noise regime on measurements and a sufficiently accurate inexact retraction.

Although the infeasible algorithms avoid the use of retraction, thereby enhancing efficiency to some extent and enabling parallelisation. However,  the efficiency of~\cite{Li2020Efficient,peng2026ns}  remains hindered by the multiple iterations used for approximating a retraction. Although Newton-Schulz iterations have recently been used to construct an inexact retraction for optimization for orthogonal group synchronization~\cite{peng2026ns}, 
the applicability to general optimization problems on the Stiefel manifold remains largely unexplored.
Motivated by this, we propose an inexact Riemannian gradient descent algorithm with  one Newton-Schulz iteration (IRGD-StieONS) and its stochastic gradient version. The main contributions of this paper are summarized as follows:



\begin{itemize}
    \item An inexact Riemannian gradient algorithm on the Stiefel manifold is proposed, where the ``inexactness'' refers to the inexact computation of the retraction. It is shown that only one-step Newton-Schulz iteration for approximating retraction is sufficient for IRGD-StieONS to guarantee global convergence and local linear convergence.
    \item Compared to the landing method~\cite{Pierre2022Fastlanding, gao2022optimization, ablin2024infeasible, SYB2024Local,sun2024retraction} and its variants~\cite{goyens2026riemannian,si2026unified, xiong2026second} as well as augmented Lagrangian-based algorithms~\cite{BinGao2019paraOrth}, the proposed algorithm IRGD-StieONS allows the use of an adaptive step size to guarantee global convergence AND local linear convergence. Such flexibility enables the use of more practical initial step sizes, e.g., the BB (Barzilai-Borwein) step size, see details in Remark~\ref{stepsize}. 
\item Using one-step Newton-Schulz iteration avoids enforcing the constraint through a penalty term. Therefore, the proposed method avoids introducing additional penalty parameters. Note that the performance of the penalty-based methods, including the landing method and its variants, augmented Lagrangian-based algorithms, and PenC and ExPen, is sensitive to the coefficient of the penalty term, whereas the proposed method avoids this issue.
Specifically, the local linear convergence rate of IRGD-StieONS depends on the condition number of the Riemannian Hessian of the minimizer that the iterates converge to (see Remark~\ref{remark2}), which coincides with the Riemannian steepest descent algorithm. Therefore, the infeasibility in the proposed algorithm does not influence the local convergence rate. However, existing infeasible algorithms~\cite{xiao2022class, xiao2024solving,SYB2024Local,sun2024retraction} exhibit a dependence of the local convergence rate on the penalty parameter. A small penalty parameter may provide insufficient constraint enforcement, resulting in slow reduction of the infeasibility, while an excessively large penalty parameter makes the local convergence factor worse and impairs the local convergence rate. See Remark~\ref{remark3} for more details.
    These theoretical advantages stem from different analysis techniques compared to those in previous work~\cite{SYB2024Local, ablin2024infeasible,sun2024retraction}.
    \item A stochastic gradient version of IRGD-StieONS (IRSGD-StieONS) is also proposed. Moreover, it shows that IRSGD-StieONS with mild assumptions has the same convergence rate as the existing Riemannian stochastic gradient descent with decreasing step size~\cite{bonnabel2013stochastic}.

    \item Numerical experiments are designed to test the IRGD-StieONS and IRSGD-StieONS algorithm, and compare it with other algorithms  to verify the superiority and robustness of the proposed algorithm.
\end{itemize}

This paper is organized as follows. Section~\ref{Sec2} summarizes notations and preliminaries used throughout the paper.
Section~\ref{Sec3} presents the framework of algorithm IRGD-StieONS and the convergence analysis of algorithm IRGD-StieONS. Section~\ref{Sec4} introduces a stochastic gradient version of IRGD-StieONS (IRSGD-StieONS). Numerical experiments are presented in Section~\ref{Sec5} to evaluate the performance of our proposed methods. Finally, we give some conclusions and future directions in Section~\ref{Sec6}.
\section{Notation and preliminaries}\label{Sec2}

Throughout this paper, $\mathbb{R}^{n\times m}$ denotes the space of real $n \times m$ matrices. The transpose of a vector $x\in \mathbb{R}^{n}$ (or a matrix $A\in\mathbb{R}^{n\times m}$) is denoted by $x^T$ (or $A^T$, respectively). For a matrix $A\in \mathbb{R}^{n\times m}$, the Frobenius norm is denoted by $\|A\|_{\F}$. For a square matrix $A$, the symmetric and skew-symmetric parts 
are defined respectively as $\operatorname{sym}(A) = \frac{A + A^T}{2}$ and $\operatorname{skew}(A) = \frac{A - A^T}{2}$. 
The Frobenius inner product of two matrices 
$A$, $B\in\mathbb{R}^{n\times m}$ is defined as $\left\langle A,B \right\rangle_{\F} = \trace (A^T B)$.
For any $a\in(0,+\infty)$, we define $\frac{a}{0} = +\infty$. For two positive sequences $\{a_k\}$, $\{b_k\}$, and $\{c_k\}$, we write $a_k = \Theta(b_k)$ if there exist positive constants $d_1,d_2, K$ such that $d_1 b_k \le a_k \le d_2 b_k, \forall k \ge K,$ and $a_k = O(c_k)$ if there exist positive constants $d_3$, $K^{'}$ such that $a_k\leq d_3c_k, \forall k \ge K^{'}.$

The Stiefel manifold $\mathrm{St}(p,n) = \{X\in \mathbb{R}^{n\times p}:X^TX = I_p\}$ is an embedded submanifold of $\mathbb{R}^{n\times p}$ of dimension $np-\frac{p(p+1)}{2}.$ The Riemannian geometry and notation are consistent with~\cite{boumal2023intromanifolds}. The tangent space $\T_X\St(p,n)$ at $X$ is given by
$\T_X\St(p,n)= \{V\in \mathbb{R}^{n\times p}:X^TV+V^TX = 0 \}.$
The tangent bundle $\T\St(p,n)$ is the disjoint union of all tangent spaces. The Riemannian metric of $\St(p, n)$ is chosen to be the canonical metric~\cite{Alan1998geometry}
\begin{equation} \label{Rmetric}
\langle U, V \rangle_X = \langle U,(I_n-\frac{1}{2} X X^T)V\rangle_{\F}, \; \forall \; X \in \St(p, n), U,V \in \T_X \St(p, n). 
\end{equation} 
The induced norm in $\T_X \St(p, n)$ is given by $\|\cdot\|_X = \sqrt{\langle \cdot, \cdot \rangle_X}$. 
A retraction $R$ is a smooth mapping $R:\T\St(p,n)\to \St(p,n)$, which   satisfies (i) $R_X(0) = X$ for all $X\in\St(p,n)$ and (ii) $\D R_X(0)[V] = V$ for all $V\in \T_X\St(p,n)$. 
A polar-based retraction given in~\cite[Section 7.3]{boumal2023intromanifolds} is given by
\begin{equation}\label{polar}
    R_X(V) = (X+V)(I_p+V^TV)^{-\frac{1}{2}},
\end{equation}
where $(X+V)(I_p+V^TV)^{-\frac{1}{2}}$ is the polar factor of $X + V$. Note that the polar-based retraction satisfies $R_X(V) = \P_{\St(p, n)} (X + V) =\mathop{\arg\min}\limits_{Y\in\St(p,n)}\|Y-(X+V)\|_{\mathrm{F}}$.
An iterative method for computing the polar factor of a matrix $Z$ is the Newton-Schulz iteration, defined by $\Psi(Z) = \frac{1}{2}Z(3I_p-Z^TZ).$ It has been shown in~\cite{bjorck1971iterative} that if $\|I-Z^TZ\|_{\F}<1,$ then $\Psi^{(\infty)}(Z) = \P_{\St(p,n)}(Z),$ where $\Psi^{(\infty)} (Z) = \underbrace{\Psi\circ \Psi\circ \dots\circ  \Psi }_{\infty}(Z).$ Consequently, the Newton-Schulz iteration provides an approach to approximately compute the polar-based retraction.


It has been shown in~\cite{NB2018Global} that there exists a constant $\kappa$ such that the inequality
\begin{equation}\label{kappa}
    \|R_X(V)-X\|_{\F}\leq \kappa  \|V\|_{\F}
\end{equation}
holds for all $X\in\St(p,n)$ and $V\in \T_X\St(p,n).$ Throughout this paper, the retraction $R$ denotes the polar-based retraction.

For a smooth function $f:\St(p,n)\to \mathbb{R}$, the Riemannian gradient $\grad f$ is defined as the unique
tangent vector satisfying 
\begin{equation*}
    \D f(X)[V] = \left\langle X, \grad f(X)\right\rangle_X , \forall \; V\in \T_X\St(p,n),
\end{equation*}
where $\D f(X)[V]$ denotes the directional derivative of $f$ at $X$ in the direction $V$. 
Therefore, the Riemannian gradient with respect to the Riemannian metric~\eqref{Rmetric} is given by
\begin{equation} \label{eq01}
\grad f(X) = \mathrm{skew}(\nabla f(X)X^T)X, \; \forall \; X \in \St(p, n),
\end{equation}
where $\nabla f(X)$ denotes the Euclidean gradient of $f$ at $X.$

We end this section by stating a technical lemma that will be used in Theorem~\ref{IRGDconveroneNS}. 
\begin{Lemma}\cite{khanh2024fundamental}\label{Lemma1}
    Let $\{a_k\},\{b_k\},\{c_k\}$ be sequences of nonnegative numbers satisfying the conditions
    \begin{equation*}\label{Lemma11}
        \begin{aligned}
           & a_{k+1} - a_k\leq b_ka_k+c_k \text{ for sufficient large } k\in \mathbb{N}\\
        & \{b_k\} \text{ is bounded, }\sum_{k=1}^{\infty}b_k = \infty, \sum_{k+1}^{\infty}c_k <\infty \text{ and } \sum_{k=1}^{\infty}b_ka_k^2 < \infty. 
        \end{aligned}
    \end{equation*}
    Then we have $a_k\to 0$ as $k\to \infty.$
\end{Lemma}

\section{An inexact Riemannian gradient descent algorithm with one  Newton-Schulz iteration (IRGD-StieONS)}\label{Sec3}
In this section, we propose an inexact Riemannian gradient descent algorithm with one Newton-Schulz iteration (IRGD-StieONS) stated in Algorithm~\ref{alg:IRGDoneNS}. 
\begin{algorithm}[!ht]
    \caption{An Inexact Riemannian Gradient Descent algorithm on the Stiefel manifold with One Newton-Schulz iteration (IRGD-StieONS) }
    \label{alg:IRGDoneNS}
\renewcommand{\algorithmicrequire}{\textbf{Input:}}
    \renewcommand{\algorithmicensure}{\textbf{Output:}}
    \begin{algorithmic}[1]
        \REQUIRE An initial point $X_0 \in \St(p,n)^{0.5}$; a positive sequence $\{t_k\}$ and a scalar $\Bar{\alpha}>0$ for initial step size selection; and a sequence $\{\gamma_k\}$ and two scalars $\sigma \in (0, 1)$ and $\lambda \in (0,1)$ for line search conditions;  
        \FOR {$k = 0, 1, 2, \ldots$}
        \STATE \label{alg:st01} $\widehat{g}_k =  \mathrm{skew}(\nabla f(X_k)X_k^T)X_k$ and $ \widehat{d}_k= - \widehat{g}_k$;
        \STATE \label{alg:st02} $\alpha_k = \Bar{\alpha}_k =  \min(\Bar{\alpha},\frac{t_k}{2\|\widehat{g}_k\|_{X_k}})$;
         \WHILE{ $f(\Psi(X_k+\alpha_k\widehat{d}_k)) > f(X_k) + \sigma \alpha_k \left\langle\hat{g}_k,\hat{d}_k\right\rangle_{X_k} + \gamma_k$} \label{alg:st03}
          \STATE   $\alpha_k = \lambda \alpha_k$; \label{alg:st04}
         \ENDWHILE \label{alg:st05}
        \STATE Set $X_{k+1}=\Psi(X_k+\alpha_k\widehat{d}_k);$ \label{alg:st06}
        \ENDFOR
    \end{algorithmic}
\end{algorithm}

IRGD-StieONS does not require all iterates to stay on the manifold. Therefore, the initial $X_0$ only needs to lie in a neighbourhood of the Stiefel manifold, i.e., $X_0\in \St(p,n)^{0.5}:= \{X\in\mathbb{R}^{n\times p}\mid \|X^TX-I\|_{\F}\leq 0.5\}$. 
Step~\ref{alg:st01} computes an extension of the Riemannian gradient $\grad f:\St(p, n) \rightarrow \mathrm{T} \St(p, n)$ at $X_k$ 
by defining the mapping
$\widehat{g}: \mathbb{R}^{n \times p} \to \mathbb{R}^{n \times p}: X \mapsto
\widehat{g}(X) = \mathrm{skew}(\nabla f(X)X^T)X.
$
For simplicity, we write $\widehat{g}_k:=\widehat{g}(X_k)$.
Step~\ref{alg:st02} sets the initial step size to $\Bar{\alpha}_k = \min(\Bar{\alpha},\frac{t_k}{2\|\widehat{g}_k\|_{X_k}})$, 
where $\Bar{\alpha}>0$ is a constant that provides a fixed upper bound for all initial step sizes 
and the sequence $\{t_k\}$ is a positive sequence that converges sublinearly to 0. Note that we define $t_k / \|\hat{g}_k\|_{X_k} = + \infty$ in the case of $\|\hat{g}_k\|_{X_k}=0$.
Steps~\ref{alg:st03} to~\ref{alg:st05} select a step size $\alpha_k$ via a backtracking line search procedure to satisfy the approximate line search condition 
\begin{equation}\label{oneNSA-G}
f(\Psi(X_k+\alpha_k\widehat{d}_k))\leq f(X_k) + \sigma \alpha_k \left\langle\hat{g}_k,\hat{d}_k\right\rangle_{X_k} + \gamma_k,
 \end{equation}
where
the sequence $\{\gamma_k\}$ is a nonnegative sequence that converges to zero, see Assumption~\ref{assumgamma}. This approximate line search condition~\ref{oneNSA-G} is inspired by the work in~\cite{Ztt2023INCG}. 
The new iterate is computed by one Newton-Schulz iteration, as shown in
Step~\ref{alg:st06}.

Algorithm~\ref{alg:IRGDoneNS} can be interpreted as an inexact Riemannian gradient descent method on the Stiefel manifold, where the retraction is replaced by one Newton–Schulz iteration. 

 \subsection{Global Convergence Analysis}\label{conver global}





The global convergence of Algorithm~\ref{alg:IRGDoneNS} relies on Assumptions~\ref{assumgradL}, \ref{assum1}, and~\ref{assumgamma}.
 \begin{Assumption}\label{assumgradL}
      The gradient of the function $f$ is Lipschitz continuous with Lipschitz constant $\Bar{L}$ on $\St(p,n)^{0.5}$, i.e., for any $X_1,X_2 \in \St(p,n)^{0.5}$, $\|\nabla f(X_1)-\nabla f(X_2)\|_{\F}\leq \bar{L}\|X_1-X_2\|_{\F}$. 
 \end{Assumption}
 If $f \in C^2$, then $f$ is Lipschitz continuously differentiable on any compact set. Since $\St(p, n)^{0.5}$ is a compact set, any $C^2$ function $f$ satisfies Assumption~\ref{assumgradL}.
 \begin{Assumption}\label{assum1}
    The sequence $\{t_k\}$ is a positive sequence that satisfies
    $$
    \lim_{k\to \infty}t_k/t_{k-1} = 1, \sum_{k=0}^{\infty}t_k = \infty,\quad  \sum_{k=0}^{\infty}t_k^4<\infty, \hbox{ and } t_k \leq \lambda\sqrt{0.2} \quad \forall k \geq 0.
    $$ 
 \end{Assumption}
 Assumption~\ref{assum1} is mild, as it holds when 
 $t_k = \frac{1}{\varpi (k+1)^p}$ for any $p \in ( \frac{1}{4}, 1]$ and a sufficiently large constant~$\varpi$. 

Defines a sequence $\{\delta_k\},$ where $\delta_k=\|I-X_k^TX_k\|_{\F}$. The sequence $\{\delta_k\}$ quantifies the feasibility violation with respect to the Stiefel manifold. The sequence $\{t_k\}$ defines a sequence $\{\theta_k\}$ by the recurrence formula $\theta_{k+1} = (\theta_k + t_k^2)^2$ and $0\leq \delta_0 \leq \theta_0 = \frac12$. As shown in Lemma~\ref{B117}, the iterates $\{X_k\}$ go to $\St(p, n)$ as $k \rightarrow \infty$. Moreover, Lemma~\ref{B117} shows that $\delta_k =\|I-X_k^TX_k\|_{\F}\leq\theta_k\leq  0.5$, which yields that the eigenvalues of $X_k^TX_k$ lie in $[\frac{1}{2},\frac{3}{2}].$ Since $X_kX_k^T$ and $X_k^TX_k$ have the same nonzero eigenvalues, we have that the eigenvalues of
$I-\frac{1}{2}X_kX_k^T$ lie in $[\frac{1}{4},1].$ Consequently, the metric $\left\langle \cdot , \cdot \right\rangle_{X_k}$ is well-defined and $\|\cdot\|_{X_k}$ is equivalent with $\|\cdot\|_{\F}$, i.e.,
\begin{align}
\frac{1}{2}\|Y\|_{\F}\leq \|Y\|_{X_k}\leq \|Y\|_{\F}, \; \forall Y\in\mathbb{R}^{n\times p}.\label{eq09}
\end{align}

Lemma~\ref{B117} is used in Lemmas~\ref{lemmaxx} and Theorems~\ref{oneNSstep},~\ref{alpha1},~\ref{conge rate}, and its proof is given in Appendix~\ref{appendix1}.
 \begin{Lemma}\label{B117}
Suppose Assumption~\ref{assum1} holds. Then $\{\theta_k\}$ satisfies
\begin{align}
    \theta_k\leq 0.5,  
\quad \theta_k = \Theta(t_k^4),
\hbox{ and } 
\sum_{k=0}^{\infty} \theta_k < \infty.\label{eq15}
\end{align}
Moreover, $\theta_k$ is an upper bound of $\delta_k$ and the distance between the iterate $X_k$ and $\St(p,n)$, i.e.
\begin{equation}\label{r7}
    \begin{aligned}
     \|X_k-\P_{\St(p,n)}(X_k)\|_{\F}&\leq \delta_k \leq 
     \theta_k, \forall k .
\end{aligned}
\end{equation}
\end{Lemma}

 

Before stating Assumption~\ref{assumgamma}, we need to introduce some constants. (i) Since the function $f$ is continuously differentiable, it is Lipschitz continuous on the compact set $\St(p,n)^{0.5}$, with the Lipschitz constant denoted by $L$. (ii) Since $\widehat{g}$ is a continuous function, there exists an upper bound constant of $\|\widehat{g}\|_{\F}$ on the compact set $\St(p, n)^{0.5}$, i.e., 
\begin{align}\label{M_g}
    M_g: = \sup_{X\in \St(p,n)^{0.5}}\|\widehat{g}(X)\|_{\F}
\end{align}
is finite. (iii) By Assumption~\ref{assumgradL}, the function $\widehat{g}$ is Lipschitz continuous on $\St(p, n)^{0.5}$, i.e. there exists a constant $L_g > 0$ such that
\begin{equation}\label{L1}
         \|\widehat{g}(X_1)-\widehat{g}(X_2)\|_{\F}\leq L_g\|X_1-X_2\|_{\F}, \forall X_1,X_2\in\St(p,n)^{0.5}.
\end{equation}
(iv) By Assumption~\ref{assumgradL} and~\cite[Lemma 2.7]{NB2018Global}, there exists constant $L^{'}$ such that $f$ is $L^{'}$-retraction-smooth with respect to the retraction $R$ in $\St(p,n),$ i.e.,
\begin{equation}\label{Lrgd}
    f(R_{X}(V))\leq f(X) + \left\langle\grad f(X), V \right\rangle_{X}+\frac{L^{'}}{2}\|V\|_X^2, \forall X\in\St(p,n), V\in \T_X\St(p,n).
\end{equation}
(v) 
Since $\P_{\St(p,n)}$ is a smooth function, it is Lipschitz continuous on the compact set $\Omega:=\St(p,n)^{0.7}\cup \{X+V \mid X\in \St(p,n), V\in \T_X\St(p,n), \|X+V\|_2\leq 1 +\frac{\Bar{\alpha}}{\lambda}M_g\}$ with a Lipschitz constant denoted by $L_p$. Now, we are ready to state Assumption~\ref{assumgamma}.
 \begin{Assumption} \label{assumgamma} 
The sequence $\{\gamma_k\}$ satisfies that
$$\gamma_k\geq  a_1\delta_k +\frac{L}{\lambda^4}\zeta_{k+1} \text{ and } \gamma_k  = O(\delta_k +\min\{\theta_{k+1},(\delta_k+\Bar{\alpha}^2\|\widehat{g}_k\|_{\F}^2)^2 \}).$$
where $ a_1 = \max\{ \frac{\max\{1,\Bar{\alpha}L^{'}\}2\Bar{\alpha}L_gM_g +\lambda^2L + 3\Bar{\alpha}M_g^2+3\Bar{\alpha}^2M_g^2L^{'}}{\lambda^2}, 3L\}$ and $ \zeta_{k+1} =\min\{\theta_{k+1}, (\delta_k+\Bar{\alpha}^2\|\widehat{g}_k\|_{\F}^2 )^2\} +L_{p}(1+\Bar{\alpha}L_g)\delta_k$. 
 \end{Assumption}

We can choose 
\begin{align}
    \gamma_k = (a_1 +\frac{LL_p}{\lambda^4}(1+\alpha_kL_g))\delta_k + \frac{L}{\lambda^4}\min\{\theta_{k+1},(\delta_k +4\Bar{\alpha}^2\|\widehat{g}_k\|_{X_k}^2)^2\}\label{gammakk}
\end{align}
 to satisfy Assumption~\ref{assumgamma},
since
\begin{align}
    a_1\delta_k+\frac{L}{\lambda^4}\zeta_{k+1}&\leq (a_1 +\frac{LL_p}{\lambda^4}(1+\alpha_kL_g))\delta_k + \frac{L}{\lambda^4}\min\{\theta_{k+1},(\delta_k +4\Bar{\alpha}^2\|\widehat{g}_k\|_{X_k}^2)^2\}\nonumber\\
    &\leq (a_1 +\frac{LL_p}{\lambda^4}(1+\alpha_kL_g))\delta_k + \frac{16L}{\lambda^4}\min\{\theta_{k+1},(\delta_k +\Bar{\alpha}^2\|\widehat{g}_k\|_{\F}^2)^2\}. \nonumber
\end{align}

For simplicity, we use $\dot{X}_k$, $\dot{g}_k$, and $\dot{d}_k$ to denote $\P_{\St(p,n)}(X_{k-1}+ \alpha_{k-1}\widehat{d}_k)$, $ \widehat{g}(\dot{X}_k)$, and $-\dot{g}_k$, respectively. 
Since $\dot{X}_k \in \St(p, n)$, we have $\widehat{g}(\dot{X}_k) = \grad f(\dot{X}_k)$. Lemma~\ref{lemmaxx} provides bounds on the deviation between $X_k$ and $\dot{X}_k$, between their corresponding gradients, and further characterizes how these deviations affect the corresponding so-called inexact retraction steps.
Lemma~\ref{lemmaxx} is used in Theorems~\ref{oneNSstep},~\ref{alpha1},~\ref{IRGDconveroneNS},~\ref{conge rate}, and its proof is given in Appendix~\ref{proofoflemmaxx}.

\begin{Lemma}\label{lemmaxx}
Suppose Assumptions~\ref{assumgradL} and~\ref{assum1} hold. Then for any $k$, it holds that
\begin{align}
    \|X_k-\dot{X}_k\|_{\F}&\leq \delta_k;\label{b}\\
    \|\widehat{g}_k-\dot{g}_k\|_{\F}&\leq L_g\delta_k;\label{g-g}\\
\|\Psi({X_k}+\frac{\alpha_k}{\lambda}\hat{d}_k)-R_{\dot{X}_k}(\frac{\alpha_k}{\lambda}\dot{d}_k)\|_{\F}&\leq \frac{1}{\lambda^4}\zeta_{k+1}; \label{zeta}  \; \hbox{ and }\\
\|\Psi({X_k}+\alpha_k\hat{d}_k)-R_{\dot{X}_k}(\alpha_k\dot{d}_k)\|_{\F}&\leq \zeta_{k+1}. \label{eq04}
\end{align}
\end{Lemma}
Theorem~\ref{oneNSstep} shows that the backtracking line search terminates in finite steps. More precisely, there exists a nonnegative integer $l_k$ such that $\alpha_k = \Bar{\alpha}_k\lambda^{l_k}$ satisfies the line search condition~\eqref{oneNSA-G}.
\begin{theorem}\label{oneNSstep}
Suppose Assumptions~\ref{assum1} and~\ref{assumgamma} hold. Then the approximate line search condition \eqref{oneNSA-G} is satisfied in a finite number of steps.
\end{theorem}
\begin{proof}

For any $k$, define a function $\Phi: \mathbb{R}\rightarrow \mathbb{R} $ by
    \begin{equation*}
        \Phi(\alpha) = f(\Psi(X_k+\alpha\widehat{d}_k)) - f(X_k)-\sigma \alpha \left\langle\hat{g}_k,\widehat{d}_k\right\rangle_{X_k} -\gamma_k.
    \end{equation*}
    By the trigonometric inequality, we have
    \begin{align}
            \|\Psi(X_k)-\dot{X}_k\|_{\F}&\leq\|\Psi(X_k)-X_k\|_{\F}+\|X_k-\dot{X}_k\|_{\F} \leq \frac{1}{2}\|X_k\|_2\|I-X_k^TX_k\|_{\F} +\|X_k-\dot{X}_k\|_{\F} \nonumber \\
            &\leq  2\delta_k , \label{o1}
    \end{align}
    where the third inequality follows from Lemmas~\ref{B117} and~\ref{lemmaxx} and that $X_k \in \St(p, n)^{0.5}$ implies $\|X_k\|_2\leq 2$.
Since $f$ is Lipschitz continuous with constant $L$, it follows from~\eqref{o1} and~\eqref{b} that
\begin{align*}
\Phi(0) = f(\Psi(X_k)) - f(X_k) -\gamma_k \leq (f(\dot{X}_k) + 2L\delta_k )-(f(\dot{X}_k)-L\delta_k)-\gamma_k < 0,
\end{align*}
where the last inequality follows from $\gamma_k \geq 3L\delta_k$ implied by Assumption~\ref{assumgamma}.
   Therefore, since $\Phi$ is continuous as a consequence of the continuity of both $f$ and $\Psi$, there exists $\alpha_k'>0$ such that $\Phi(\alpha)\leq 0$ for all $\alpha\in [0, \alpha^{'}_k]$. 
    Therefore,~\eqref{oneNSA-G} holds for all $\alpha\in [0, \alpha^{'}_k]$, which completes the proof.
\end{proof} 
Theorem~\ref{alpha1} gives a lower bound of the accepted step size $\alpha_k$.
\begin{theorem}\label{alpha1}
    Suppose Assumptions~\ref{assumgradL}, \ref{assum1}, and \ref{assumgamma} hold. 
    Let $\alpha_k$ be the accepted step size in the $k$th
iteration of Algorithm~\ref{alg:IRGDoneNS}. Then we have
\begin{equation}\label{B1}
    \alpha_k\geq \min\{\Bar{\alpha}_k,\frac{2\lambda}{L^{'}} (1-\sigma)\}  = \min\{\Bar{\alpha},\frac{t_k}{2\|\widehat{g}_k\|_{X_k}}, \frac{2\lambda}{L^{'}} (1-\sigma)\} \quad \forall k, \text{ and }\sum_{k = 0}^{\infty}\alpha_k = \infty.
\end{equation}
    
\end{theorem}
\begin{proof}
      By Theorem \ref{oneNSstep}, there exists an integer $l_k\geq 0$ such that $\alpha_k = \Bar{\alpha}_k\lambda^{l_k} >0 $, where $l_k$ is the smallest nonnegative integer satisfying~\eqref{oneNSA-G}. We distinguish the following cases.
      
      \medskip
    \noindent
    \textbf{Case 1:} $\|\widehat{g}_k\|_{X_k} = 0.$ In this case, $\Bar{\alpha}_k = \Bar{\alpha}$ and together with~\eqref{eq09} gives $\|\widehat{g}_k\|_{\F} = 0.$  Since $f$ is Lipschitz continuous with constant $L$, we obtain
    \begin{align}     f(\Psi(X_k+\Bar{\alpha}_k\widehat{d}_k))\leq& f(X_k) + L\|\Psi(X_k+\Bar{\alpha}_k\widehat{d}_k)-X_k\|_{\F}\nonumber\\
\leq& f(X_k) +L (\Bar{\alpha}_k\|\widehat{g}_k\|_{\F} + \frac{1}{2}\|X_k\|_2\|X_k^TX_k-I_p\|_{\F} + \frac{1}{2}\Bar{\alpha}_k^2\|X_k\|_2\|\widehat{g}_k\|_{\F}^2\nonumber\\
&+\frac{\Bar{\alpha}_k}{2}\|\widehat{g}_k\|_{\F}\|X_k^TX_k-I_p\|_{\F}+ \frac{1}{2}\Bar{\alpha}_k^3\|\widehat{g}_k\|_{\F}^3) \text{ (by triangle inequality)}\nonumber\\
\leq &f(X_k) +L \delta_k \leq f(X_k) + \gamma_k ,\label{r8}
    \end{align}
    where the third inequality follows from Lemma~\ref{B117} and $\|\widehat{g}_k\|_{\F} =0$, the last inequality follows from Assumption~\ref{assumgamma}.
   Therefore, condition~\eqref{oneNSA-G} is satisfied with $\alpha_k=\bar{\alpha}_k$, implying $\alpha_k = \Bar{\alpha}_k>0.$
    \medskip
    \noindent
    \textbf{Case 2:} $\|\widehat{g}_k\|_{X_k} \neq 0$. If $l_k= 0$, then $\alpha_k = \Bar{\alpha}_k>0$ which yields the inequality in~\eqref{B1}. Next, we suppose 
    $l_k\geq 1$.
    In this case, the previous trial step size $\frac{\alpha_k}{\lambda}$ fails to satisfy~\eqref{oneNSA-G}, yielding the inequality
\begin{equation}\label{istep}
       f(\Psi(X_k+\frac{\alpha_k}{\lambda}\widehat{d}_k))\geq f(X_k)+\sigma \frac{\alpha_k}{\lambda} \left\langle\hat{g}_k,\widehat{d}_k\right\rangle_{X_k} + \gamma_k.
\end{equation}
By Lemma~\ref{lemmaxx} and $f$ is Lipschitz continuous with Lipschitz constant $L$, we have
\begin{align}
    f(R_{\dot{X}_k}(\frac{\alpha_k}{\lambda}\dot{d}_k))\geq f(\Psi(X_k+\frac{\alpha_k}{\lambda}\widehat{d}_k)) -  \frac{L}{\lambda^4}\zeta_{k+1},
    f(\dot{X}_k)\leq f(X_k)+L\delta_k .\nonumber
\end{align} 
Combining these with~\eqref{Lrgd} gives
    \begin{align}
        f(\Psi(X_k+\frac{\alpha_k}{\lambda}\widehat{d}_k))&\leq f(X_k) + \frac{\alpha_k}{\lambda}\left\langle\dot{d}_k, \dot{g}_k\right\rangle_{\dot{X}_k}+\frac{\alpha_k^2}{2\lambda^2}L^{'}\|\dot{d}_k\|_{\dot{X}_k}^2+\frac{L}{\lambda^4}\zeta_{k+1} + L\delta_k \label{r4}
    \end{align}
By Lemma~\ref{lemmaxx} and~\eqref{eq09}, we have $\|\widehat{g}_k-\dot{g}_k\|_{\dot{X_k}}\leq\|\widehat{g}_k-\dot{g}_k\|_{\F}\leq L_g\delta_k .$ It follows that 
\begin{gather}
     \|\widehat{g}_k\|_{\dot{X}_k}-L_g\delta_k\leq\|\dot{g}_k\|_{\dot{X_k}}\leq \|\widehat{g}_k\|_{\dot{X_k}}+L_g\delta_k, \quad \text{ (by~\eqref{g-g})} \label{w1} \\
\|\widehat{g}_k\|_{\dot{X_k}}\|\dot{g}_k\|_{\dot{X_k}}-L_g\delta_k\|\dot{g}_k\|_{\dot{X_k}}
     \leq\|\dot{g}_k\|_{\dot{X_k}}^2
     \leq \|\widehat{g}_k\|_{\dot{X_k}}\|\dot{g}_k\|_{\dot{X_k}}+L_g\delta_k\|\dot{g}_k\|_{\dot{X_k}}, \hbox{ and}
     \text{ (\eqref{w1} times $\|\dot{g}_k\|_{\dot{X_k}}$)} \nonumber\\
     \|\widehat{g}_k\|_{\dot{X_k}}^2 -L_g\delta_k\|\widehat{g}_k\|_{\dot{X_k}}-L_g\delta_k\|\dot{g}_k\|_{\dot{X_k}}
     \leq\|\dot{g}_k\|_{\dot{X_k}}^2
     \leq \|\widehat{g}_k\|_{\dot{X_k}}^2+L_g\delta_k\|\widehat{g}_k\|_{\dot{X_k}}+L_g\delta_k\|\dot{g}_k\|_{\dot{X_k}}.
    \text{(by \eqref{w1})} \label{w3}
\end{gather}
Inserting~\eqref{w3} into~\eqref{r4} gives 
    \begin{align}
        &f(\Psi(X_k+\frac{\alpha_k}{\lambda}\widehat{d}_k))\leq f(X_k) - \frac{\alpha_k}{\lambda}(\|\widehat{g}_k\|_{\dot{X}_k}^2-L_g\delta_k\|\widehat{g}_k\|_{\dot{X}_k}- L_g\delta_k\|\dot{g}_k\|_{\dot{X}_k} )\nonumber\\
        &+\frac{\alpha_k^2}{2\lambda^2}L^{'}(\|\widehat{g}_k\|_{\dot{X}_k}^2+L_g\delta_k\|\widehat{g}_k\|_{\dot{X}_k}+L_g\delta_k\|\dot{g}_k\|_{\dot{X}_k})+\frac{L}{\lambda^4}\zeta_{k+1} + L\delta_k\nonumber\\
        =& f(X_k) - \frac{\alpha_k}{\lambda}\|\widehat{g}_k\|_{\dot{X}_k}^2 + \frac{\alpha_k^2}{2\lambda^2}L^{'}\|\widehat{g}_k\|_{\dot{X}_k}^2 + \frac{L}{\lambda^4}\zeta_{k+1} + L\delta_k\nonumber \\
        & + \frac{\alpha_kL_g(\|\widehat{g}_k\|_{\dot{X}_k}+\|\dot{g}_k\|_{\dot{X}_k}) }{\lambda}\delta_k + \frac{\alpha_k^2L^{'}L_g(\|\widehat{g}_k\|_{\dot{X}_k}+\|\dot{g}_k\|_{\dot{X}_k}) }{2\lambda^2}\delta_k\nonumber\\
        \leq& f(X_k) - \frac{\alpha_k}{\lambda}\|\widehat{g}_k\|_{\dot{X}_k}^2 + \frac{\alpha_k^2}{2\lambda^2}L^{'}\|\widehat{g}_k\|_{\dot{X}_k}^2 + \frac{L}{\lambda^4}\zeta_{k+1} + L\delta_k+ \frac{\max\{1,\Bar{\alpha}L^{'}\}2\Bar{\alpha}L_gM_g}{\lambda^2}\delta_k,\label{r6}
    \end{align}
where the last inequality follows from that $\lambda\geq \lambda^2, 2\lambda^2\geq \lambda^2$, $\bar{\alpha} \geq \alpha_k$, and $\|\widehat{g}_k\|_{\dot{X}_k}+\|\dot{g}_k\|_{\dot{X}_k}\leq 2 M_g$. Combining (\ref{istep}) and (\ref{r6}) yields 
\begin{align}
    -\sigma \frac{\alpha_k}{\lambda} \|\widehat{g}_k\|_{X_k}^2
    &\leq- \frac{\alpha_k}{\lambda}\|\widehat{g}_k\|_{\dot{X}_k}^2 + \frac{\alpha_k^2}{2\lambda^2}L^{'}\|\widehat{g}_k\|_{\dot{X}_k}^2 + \frac{L}{\lambda^4}\zeta_{k+1}+ L\delta_k+ \frac{\max\{1,\Bar{\alpha}L^{'}\}2\Bar{\alpha}L_gM_g}{\lambda^2}\delta_k-\gamma_k\label{eq35}
\end{align}
It follows from~\eqref{Rmetric} that
\begin{align}
\|\widehat{g}_k\|_{X_k}^2-\|\widehat{g}_k\|^2_{\dot{X}_k}&\leq \frac{1}{2} \|\widehat{g}_k\|_F^2 \|X_kX_k^T-\dot{X}_k\dot{X}_k^T\|_{\F}\leq \frac{M_g^2}{2}\|(X_k-\dot{X}_k)X^T_k+ \dot{X}_k(X_k-\dot{X}_k)^T\|_{\F}\nonumber\\
& \leq \frac{M_g^2}{2}(\|X_k\|_2+\|\dot{X}_k\|_2)\delta_k\leq \frac{3M_g^2}{2}\delta_k, \label{eq36}
\end{align}
where the last inequality follows from $\|\dot{X}_k\|_2 = 1$ and $\|X_k\|_2\leq \sqrt{1+\delta_k}<2$. Inserting~\eqref{eq36} into~\eqref{eq35} gives
\begin{align}
    -\sigma \frac{\alpha_k}{\lambda} \|\widehat{g}_k\|_{X_k}^2
    \leq&- \frac{\alpha_k}{\lambda}\|\widehat{g}_k\|_{X_k}^2+ \frac{3\Bar{\alpha}M_g^2}{2\lambda}\delta_k   + \frac{\alpha_k^2}{2\lambda^2}L^{'}\|\widehat{g}_k\|_{X_k}^2 + \frac{3\Bar{\alpha}^2M^2_g}{4\lambda^2}L^{'}\delta_k \nonumber\\
    +& \frac{L}{\lambda^4}\zeta_{k+1}+ L\delta_k+ \frac{\max\{1,\Bar{\alpha}L^{'}\}2\Bar{\alpha}L_g M_g}{\lambda^2}\delta_k-\gamma_k \nonumber\\
    \leq &-\frac{\alpha_k}{\lambda}\|\widehat{g}_k\|_{X_k}^2 + \frac{\alpha_k^2}{2\lambda^2}L^{'}\|\widehat{g}_k\|_{X_k}^2. \text{ (by Assumption~\ref{assumgamma})}\nonumber
\end{align}
Dividing both sides by $\alpha_k \|\widehat{g}_k\|_{\F}^2/\lambda >0$ gives 
\begin{equation*}
     \alpha_k \geq \frac{2\lambda}{L^{'}} (1-\sigma)>0.
\end{equation*}
It follows from the above two cases that the inequality in~\eqref{B1} holds.

Next, we prove $\sum_{k = 0}^{\infty}\alpha_k = \infty$ by contradiction. Suppose $\sum_{k = 0}^{\infty}\alpha_k < \infty$, then $\lim_{k\to\infty}\alpha_k = 0.$ Consequently, there exists a positive integer $N$ such that for all $k>N$, $\alpha_k<\min\{\Bar{\alpha},\frac{2\lambda}{L^{'}}(\frac{1}{4}-\sigma)\}$. Combining this with the inequality in~\eqref{B1} yields
$\alpha_k=\frac{t_k}{2\|\widehat{g}_k\|_{X_k}}\geq \frac{1}{2M_g}t_k, \forall k>N.$
Since $\sum_{k=0}^{\infty}t_k = \infty$, it follows that $\sum_{k = 0}^{\infty}\alpha_k = \infty$, which contradicts the assumption. Hence, $\sum_{k = 0}^{\infty}\alpha_k = \infty$.
\end{proof}

In Theorem~\ref{IRGDconveroneNS}, we present a global convergence analysis of Algorithm~\ref{alg:IRGDoneNS}.

\begin{theorem}\label{IRGDconveroneNS}
    Suppose Assumptions~\ref{assumgradL},~\ref{assum1}, and~\ref{assumgamma} hold. Then Algorithm~\ref{alg:IRGDoneNS} converges in the sense that     
    \begin{equation*}\label{g}
        \lim_{k\rightarrow \infty} \|\hat{g}_k\|_{X_k} = 0.
    \end{equation*}
    Moreover, if $X_*$ is an accumulation point of the sequence $\{X_k\}$, then $f(X_k)\to f(X_*)$. 
    Furthermore, if $ t_k = \frac{2\iota M_g}{(k+1)^{\frac{1}{4}+\ell}},$ then Algorithm~\ref{alg:IRGDoneNS} returns a point $X_k$ satisfying  
    $\|\hat{g}_k\|_{X_k}\leq \epsilon$ within at most $$((\frac{3}{4}-\ell)\frac{f(X_0)-f^*+\sum_{k=0}^{\infty}\gamma_k+\frac{\sigma \epsilon^2\iota}{3/4-\ell}}{\sigma \epsilon^2\iota})^{\frac{1}{3/4-\ell}}$$
    iterations and $\delta_k \leq \epsilon $ at most $ (\frac{\Bar{M}(2\iota M_g )^4}{\epsilon})^{\frac{1}{1+4\ell}}-1$ iteration, where $\iota = \min\{\Bar{\alpha},\frac{2\lambda}{L^{'}}(1-\sigma)\}, \ell\in(0,\frac{3}{4})$ is a constant and  $f^{*} = \inf_{X\in \St(p,n)^{0.5}} f(X)>-\infty$, $f^{*}$ naturally exists because $f$ is continuous and the set $\St(p,n)^{0.5}$ is compact. Specifically, if $\ell$ is sufficiently small, in the limit, Algorithm~\ref{alg:IRGDoneNS} returns $X_k$ with $\|\widehat{g}_k\|_{X_k}\leq \epsilon$ in $\mathcal{O}(1/\epsilon^{\frac{8}{3}})$ iterations and $\delta_k \leq \epsilon$ in $\mathcal{O}(1/\epsilon)$ iterations.
\end{theorem}

\begin{proof}
Using the approximate line search condition (\ref{oneNSA-G}) yields
\begin{align}
\sigma \sum^{n-1}_{k=0}\alpha_k \|\widehat{g}_k\|_{X_k}^2& \leq \sum_{k=0}^{n-1}(f(X_{k})-f(X_{k+1})+\gamma_k) \leq f(X_0)- f(X_n)+\sum_{k=0}^{n-1}\gamma_k\nonumber\\
&\leq f(X_0)-f^{*}+\sum_{k=0}^{\infty}\gamma_k< \infty,\label{pppp}
\end{align}
where the last inequality follows from $\gamma_k = O(\delta_k +\min\{\theta_{k+1},(\delta_k+\Bar{\alpha}^2\|\widehat{g}_k\|_{\F}^2)^2 ) = O(\theta_k+\theta_{k+1})$ and $\sum_{k =0}^{\infty}\theta_k<\infty.$
The difference of successive gradient norms satisfies
    \begin{align}
\|\widehat{g}_{k+1}\|_{\F}-\|\widehat{g}_k\|_{\F}&\leq L_g\|X_{k+1}-X_k\|_{\F} \quad \hbox{ (by~\eqref{L1})} \nonumber\\
        &\leq L_g(\|X_{k+1}-
        R_{\dot{X}_k}(\alpha_k\dot{d}_k)\|_{\F} + \|R_{\dot{X}_k}(\alpha_k\dot{d}_k) - \dot{X}_k\|_{\F}+\|\dot{X}_k-X_k\|_{\F})\nonumber\\
        &\leq L_g\zeta_{k+1} + L_g\kappa\alpha_k\|\dot{d}_k\|_{\F} + L_g\delta_k \hbox{ (by~~\eqref{eq04},~\eqref{kappa}, and~\eqref{r7})}\label{B111} \\
        & \leq L_g\zeta_{k+1} + L_g\kappa\alpha_k\|\widehat{g}_k\|_{\F} + (\kappa\Bar{\alpha} L_g+1)L_g\delta_k, \label{B3}
    \end{align}
where the last inequality follows from $\alpha_k\leq \Bar{\alpha}$ for all $k$ and $\|\dot{d}_k\|_{\F} \leq \|\widehat{g}_k\|_{\F} +L_g\delta_k$ since $\|\widehat{g}_k-\dot{g}_k\|_{\F}\leq L_g\delta_k$ and $\dot{d}_k = -\dot{g}_k$.

Since the series $\sum_{k=0}^{\infty}\theta_k$ and $\sum_{k=0}^{\infty}\zeta_k$ converge, and inequality (\ref{pppp}) together with Theorem~\ref{alpha1} ensures that $ \sum^{\infty}_{k=0}\alpha_k \|\widehat{g}_k\|_{\F}^2\leq 4\sum^{\infty}_{k=0}\alpha_k \|\widehat{g}_k\|_{X_k}^2  <\infty$ and $\sum_{k=0}^{\infty}\alpha_k = \infty$, applying Lemma~\ref{Lemma1} to (\ref{B3}) yields that $\lim_{k\to \infty}\|\widehat{g}_k\|_{\F}= 0.$ Combing this with~\eqref{eq09}, we have $\lim_{k\to \infty}\|\widehat{g}_k\|_{X_k}= 0.$


Condition~(\ref{oneNSA-G}) can be rewritten as 
\begin{equation}\label{eq10}
f(\Psi(X_k+\alpha_k\widehat{d}_k))+u_{k+1} \leq f(X_k)+u_k +\sigma \alpha_k \left\langle\hat{g}_k,\widehat{d}_k\right\rangle_{X_k},
\end{equation}
where $u_k = \sum_{i= k}^{\infty}\gamma_k$. We have that $u_k\to 0$ as $k\to \infty$ and $u_k-u_{k+1} = \gamma_k.$ It follows from~\eqref{eq10} that $\{f(X_k)+u_k\}$ is nonincreasing. Since $\inf_{k\in\mathbb{N}}f(X_k)>-\infty$ and $u_k\to 0$, the sequence $\{f(X_k)+u_k\}$ is bounded from below, and thus is convergent. Taking into account that $u_k\to 0,$ it follows that $f(X_k)$ is convergent as well. Since $X_*$ is an accumulation point of $\{X_k\}$, the continuity of $f$ implies that $f(X_*)$  is also an accumulation point of $\{f(X_k)\},$ which yields $f(X_k)\to f(X_*)$ due to that convergence of $\{f(X_k)\}.$ 
    
 
If Algorithm \ref{alg:IRGDoneNS} executes $n-1$ iterations without termination, i.e., $\|\widehat{g}_k\|_{X_k}> \epsilon$, for all $k =0,1,...,n-1$, then inequality \eqref{pppp} yields
\begin{equation}\label{A11}
\begin{aligned}
     f(X_0)-f^*+\sum_{k=0}^{\infty}\gamma_k\geq \sigma  \sum^{n-1}_{k=0}\alpha_k\|\widehat{g}_k\|_{X_k}^2 > \sigma \epsilon^2 \sum^{n-1}_{k=0}\alpha_k. 
\end{aligned}
\end{equation}
   Combining $\|\widehat{g}_k\|_{X_k}\leq \|\widehat{g}_k\|_{\F}\leq M_g$ with Theorem~\ref{alpha1} yields $\alpha_k \geq \min\{ \Bar{\alpha},\frac{t_k}{2M_g}, \frac{2\lambda}{L^{'}}(1-\sigma)\}.$ It follows from $ t_k = \frac{2\iota M_g }{(k+1)^{\frac{1}{4}+\ell}},$ that $\alpha_k\geq \frac{ \iota }{(k+1)^{\frac{1}{4}+\ell}}.$  Therefore,  
    \begin{equation}\label{A10}
\begin{aligned}
   \frac{1}{\iota}\sum_{k=0}^{n-1}\alpha_k&\geq  \int_{0}^{n}\frac{1}{(x+1)^{\frac{1}{4}+\ell}}dx  =  \frac{1}{3/4-\ell}(n+1)^{3/4-\ell}-\frac{1}{3/4-\ell}. \\
\end{aligned}
\end{equation}
Combining \eqref{A10} and \eqref{A11} gives
$
     f(X_0)-f^*+\sum_{k=0}^{\infty}\gamma_k> \sigma \epsilon^2 \iota (\frac{1}{3/4-\ell}(n+1)^{3/4-\ell}-\frac{1}{3/4-\ell}).
$
By contradiction, the algorithm must terminate if $n \geq ((\frac{3}{4}-\ell)\frac{f(X_0)-f^*+\sum_{k=0}^{\infty}\gamma_k+\frac{\sigma \epsilon^2\iota}{3/4-\ell}}{\sigma \epsilon^2\iota})^{\frac{1}{3/4-\ell}}.$ 
By Lemma~\ref{B117}, there exists a constant $\Bar{M}$ such that  $\delta_k\leq \theta_k \leq \Bar{M}t_k^4 = \frac{\Bar{M}(2\iota M_g )^4 }{(k+1)^{1+4\ell}} .$ Therefore, $\delta_k \leq \epsilon$ if $n\geq (\frac{\Bar{M}(2\iota M_g )^4}{\epsilon})^{\frac{1}{1+4\ell}}-1.$

\end{proof}
\subsection{Local Convergence Rate Analysis}\label{rate}
The local convergence rates analysis of Algorithm~\ref{alg:IRGDoneNS} relies on Assumption~\ref{PL}. 
\begin{Assumption}\label{PL}
    (Local Riemannian P$\text{\L}$ Condition~\cite{SYB2024Local}) The function $f: \mathbb{R}^{n\times p}\to \mathbb{R}$ satisfies the local
Riemannian Polyak-$\text{\L}$ojasiewicz (P$\text{\L}$) condition on the Stiefel
manifold with a factor $\mu>0$ if 
$$|f(X)-f^{*}_{\mathcal{S}}|\leq \frac{1}{2\mu}\|\grad(X)\|_{\F}^2,$$
for any point $X\in \St(p,n)\cap \mathcal{D}(\mathcal{S},2r)$, where $r>0 $ is a constant, $\mathcal{S}$ denotes the set
of all local minimizers of $f$ over $\St(p, n)$ with a given value $f^{*}_{\mathcal{S}}$ and $\mathcal{D}(\mathcal{S},2r) =\{X \in \mathbb{R}^{n \times p} \mid \dist(\mathcal{S},X)\leq 2r\},$ with $\dist(\mathcal{S},X) = \min_{Y\in\mathcal{S}}\|X-Y\|_{\F}.$
\end{Assumption}

The local Riemannian P$\text{\L}$ condition is a relaxation of
the global Riemannian P$\text{\L}$ condition. This assumption is easier to satisfy compared to the
geodesic strong convexity. Many traditional tasks involving the
Stiefel manifold, such as the PCA problem and the generalized
quadratic problem, satisfy the local Riemannian P$\text{\L}$ condition but do not meet global Riemannian P$\text{\L}$ condition and geodesic strong convexity~\cite{SYB2024Local,liu2019quadratic}. 
Theorem~\ref{conge rate} shows that the accepted step sizes eventually have a uniform positive lower bound.
\begin{theorem}\label{conge rate}
    Let $\{X_k\}$ be the sequence generated by Algorithm \ref{alg:IRGDoneNS}. Suppose that Assumptions \ref{assumgradL},~\ref{assum1} \ref{assumgamma} and \ref{PL}  hold, that all accumulation points of $\{X_k\}$ are local minimizers, and that $t_k = \Theta(\frac{1}{k^a})$, for $a \in (\frac{1}{4}, \frac{1}{2})$. Then the step size sequence $\{\alpha_k\}$ is bounded below by a positive constant for all sufficiently large $k$,
   i.e. there exist $\alpha^{'}>0$ and $K>0$ such that $\alpha_k\geq\alpha^{'} \text{ for all } k\geq K.$
\end{theorem}
\begin{proof}
Since all accumulation points of $\{X_k\}$ share the same objective
value by Theorem~\ref{IRGDconveroneNS}, we select
$f_{\mathcal S}^{*}:=f(X_*)$ and $X_*$ is an accumulation point of $\{X_k\}.$ 
Moreover, by \eqref{b} and the fact that $\delta_k \to 0$, any accumulation point of $\{X_k\}$ is also an accumulation point of $\{\dot{X}_k\}$. We first show that the sequence
$\{\dot X_k\}$ eventually enters and remains in the neighborhood
$\mathcal D(\mathcal S,2r)$, so that the local Riemannian  P$\text{\L}$ condition becomes
applicable.
Suppose that infinitely many iterates lie outside
$\mathcal D(\mathcal S,2r)$. Then there exists a subsequence
$\{\dot X_{k_j}\}$ satisfying $\dist(\dot X_{k_j},\mathcal S)>2r.$ Since $\{\dot{X}_{k_j}\}$ is bounded, there exists a convergent subsequence $\{\dot{X}_{k_{j_l}}
\}$, denote its limit as $ X^{'}.$ 
Therefore, $X^{'}\in \mathcal{S}$ and $\dist (\mathcal{S}, X_{k_{j_l}})\leq\dist (X^{'}, X_{k_{j_l}}) \leq r$ for all sufficiently large $l$, which is a contradiction to $\dist(\mathcal{S},\dot{X}_{k_j})>2r.$
Therefore, the local Riemannian P$\text{\L}$ inequality holds for all
sufficiently large $k$, namely,
\begin{align}
f(\dot{X}_{k}) - f(X_*)& \le \frac{1}{2\mu}\|\dot{g}_{k}\|_{\F}^2,\label{s2_new}
\end{align}
for all sufficiently large $k$. Similar to~\eqref{w3}, $\|\dot{g}_k\|_{\F}^2\leq \|\widehat{g}_k\|_{\F}^2 + L-g\delta_k\|\widehat{g}_k\|_{\F}^2+L_g\delta_k\|\dot{g}_k\|^2_{\F}.$
Inserting this into~\eqref{s2_new}, we obtain
\begin{align}
&f(\dot{X}_{k}) - f(X_*)\leq \frac{1}{2\mu} (\|\widehat{g}_k\|_{\F}^2+L_g\delta_k\|\widehat{g}_k\|_{\F}+L_g\delta_k\|\dot{g}_k\|_{\F})\nonumber\\
\implies & \|\widehat{g}_k\|_{\F}^2 \geq 2\mu(f(\dot{X}_{k}) - f(X_*)) - L_g\delta_k\|\widehat{g}_k\|_{\F}-L_g\delta_k\|\dot{g}_k\|_{\F}.\label{eq11}
\end{align}
Combing~\eqref{eq09} and~\eqref{eq11} yields
\begin{align}
\|\widehat{g}_k\|_{X_k}^2\geq\frac{1}{4}\|\widehat{g}_k\|_{\F}^2 \geq\frac{1}{4}(2\mu(f(\dot{X}_k) - f(X_*)) - L_g\delta_k\|\widehat{g}_k\|_{\F} - L_g\delta_k\|\dot{g}_k\|_{\F}).\label{y1_new}
\end{align}
Inserting~\eqref{y1_new} into~\eqref{oneNSA-G} yields
\begin{align}
    f(X_{k+1}) - f(X_k)
    &\le -\frac{1}{2}\sigma\mu \alpha_k (f(\dot{X}_k) - f(X_*)) + \frac{1}{4}\sigma \alpha_k \Big(L_g\delta_k\|\widehat{g}_k\|_{\F} + L_g\delta_k\|\dot{g}_k\|_{\F} \Big) + \gamma_k\label{y4}\\
    &\le -\frac{1}{2}\sigma\mu\alpha_k (f(\dot{X}_k) - f(X_*)) + C\delta_k + \gamma_k
     \le  C\delta_k + \gamma_k,\label{d2b_new}
\end{align}
where the second inequality follows from that $C>0$ is a constant such that $\frac{1}{4}\sigma\alpha_k( L_g\delta_k\|\hat{g}_k\|_{\F} + L_g\delta_k\|\dot{g}_k\|_{\F} +2\mu L\delta_k)\leq C \delta_k$ and the last inequality follows from  $X_*$ is a local minimizer.

Next, since $t_k = \Theta(k^{-a})$ with $\frac{1}{4} < a < \frac{1}{2}$, there exist constants $T',T'',Q, E>0$  such that
\begin{align}
    T' k^{-a} \le &t_k \le T'' k^{-a}, \label{eq06} \\
    &\delta_k \leq \theta_k =  O(t_k^4)\leq  Q k^{-4a}\text{ (by Lemma~\ref{B117}) }, \label{w6}\\
    &\gamma_k =  O(\theta_k+\theta_{k+1}) \le E k^{-4a} \text{ (by Assumption~\ref{assumgamma})},\label{w5}
\end{align}
for all sufficiently large $k$. The bounds
$\delta_k=O(k^{-4a})$
and
$\gamma_k=O(k^{-4a})$
suggest that the perturbation terms decay faster than
$k^{-2a}$.
Motivated by this observation, we next establish the estimation
\[
f(X_k)-f(X_*)
=O(k^{-2a}),
\]
which is sufficient for proving the existence of a positive lower
bound on the accepted step sizes. We now prove by induction that there exists $M>0$ such that
\begin{align}\label{p0}
    f(X_k) - f(X_*) \le M k^{-2a},
\end{align}
for all sufficiently large $k$.

The claim holds for some sufficiently large $k_0$ by choosing $M$ large enough. Assume it holds for some $k \ge k_0$. Denote $T = \frac{M}{2^{2a}} - CQ - E > 0$ and then consider two cases.

\medskip
\noindent
\textbf{Case 1:}  Suppose that $f(X_k) - f(X_*) \le T k^{-4a}$.
Then, from \eqref{d2b_new}, for sufficiently large $k$, 
\begin{align}
    f(X_{k+1}) - f(X_*)
&\leq f(X_{k}) - f(X_*) +C\delta_k+\gamma_k \nonumber\\
&\leq (T + CQ + E) k^{-4a} \text{ (by the inductive hypothesis, ~\eqref{w6}, and~\eqref{w5})}\nonumber\\
&= M \left(2 k^2\right)^{-2a} 
\le M (k+1)^{-2a}. \text{(since $2k^2\geq k+1$ holds for $k\geq1$) }\nonumber
\end{align}

\medskip
\noindent
\textbf{Case 2:} Suppose instead that $f(X_k) - f(X_*) \ge T k^{-4a}$.
From~\eqref{y4},~\eqref{b}, and $f$ is Lipschitz continuous with constant $L$, we have
\begin{align}
    f(X_{k+1}) - f(X_k)&\leq  -\frac{1}{2}\sigma\mu \alpha_k (f(X_k) - f(X_*)) + \frac{1}{4}\sigma \alpha_k \Big(L_g\delta_k\|\widehat{g}_k\|_{\F} + L_g\delta_k\|\dot{g}_k\|_{\F} + 2\mu L\delta_k \Big) + \gamma_k\nonumber\\
    &\le -\frac{1}{2}\sigma\mu \underline{\alpha}_k T k^{-4a} + C\delta_k + \gamma_k,\label{d2b}
\end{align}
where $\underline{\alpha}_k := \min\{\Bar{\alpha}, \frac{t_k}{2\|\widehat{g}_k\|_{X_k}}, \frac{2\lambda}{L'}(1-\sigma)\}$.

\medskip
\noindent
\textbf{Case 2.1:} $\underline{\alpha}_k = \underline{\alpha} := \min\{\Bar{\alpha}, \frac{2\lambda}{L'}(1-\sigma)\}$.
Combining~\eqref{w6},~\eqref{w5},~\eqref{p0}, and~\eqref{d2b} yields
\begin{align}\label{d6}
    f(X_{k+1}) - f(X_*)
\le M k^{-2a} - \frac{1}{2}\sigma\mu \underline{\alpha} T k^{-4a} + (CQ + E) k^{-4a}.
\end{align}
We can adjust the parameter $M$ and let $K$ be sufficiently large such that for any $k\geq K$ we have
\begin{align}
&\frac{1}{2^{2a}}\geq \frac{2a}{\sigma\mu \underline{\alpha}k^{1-2a}}+\frac{(1+\frac{1}{2}\sigma \mu \underline{\alpha} )(CQ+E)}{\frac{1}{2}\sigma  \mu \underline{\alpha} M}\implies \frac{1}{2}\sigma \mu \underline{\alpha} T\geq 2aM\frac{1}{k^{1-2a}} + CQ + E,\nonumber
\end{align}
where ``$\implies$'' follows from $T=\frac{M}{2^{2a}} - CQ - E$.
Inserting this into \eqref{d6} yields that for a sufficiently large $k$,
\begin{equation*}
    \begin{aligned}
         f(X_{k+1})-f(X_*)&\leq M\frac{1}{k^{2a}}(1-\frac{2a}{k})\leq M(k+1)^{-2a}, 
    \end{aligned}
\end{equation*}
where the last inequality follows from $(k+1)^{-2a} = k^{-2a}(1+\frac{1}{k})^{-2a}\geq k^{-2a}(1-\frac{2a}{k}) $ (by Bernoulli's inequality).

\medskip
\noindent
\textbf{Case 2.2:} $\underline{\alpha}_k = \frac{t_k}{2\|\widehat{g}_k\|_{X_k}}$.
From \eqref{oneNSA-G} and hypothesis of induction, we have
    \begin{align}
        f(X_{k+1}) - f(X_*)&\leq M\frac{1}{k^{2a}} -\frac{1}{2}\sigma t_k\|\widehat{g}_k\|_{X_k}+\gamma_k\nonumber\\
        &\leq M\frac{1}{k^{2a}} -\frac{1}{4}\sigma t_k\|\widehat{g}_k\|_{\F}+\gamma_k \text{ (by~\eqref{eq09})}\nonumber\\
        & \leq M\frac{1}{k^{2a}} -\frac{1}{4}\sigma t_k\|\dot{g}_k\|_{\F}+\frac{1}{4}L_g\delta_k \sigma t_k+\gamma_k \; (\text{by~\eqref{g-g}} )\nonumber\\
        & \leq M\frac{1}{k^{2a}} -\frac{1}{4}\sigma t_k\sqrt{2\mu(f(X_k)- f(X_*)-L\delta_k)}+\frac{1}{4}L_g\delta_k\sigma t_k+\gamma_k\nonumber\\
        & \leq M\frac{1}{k^{2a}} -\frac{1}{4}\sigma T^{'}\sqrt{2\mu (T-LQ)}\frac{1}{k^{2a}}+\frac{1}{4}L_gQT^{''}\frac{1}{k^{4a}}+E\frac{1}{k^{4a}},\label{y3}
    \end{align}
    where the fourth inequality follows from~\eqref{s2_new}, (\ref{b}), and $f$ is Lipschitz continuous with constant $L$.
Let the parameters $M$ and $K$ be adjusted sufficiently large satisfying $M \leq K$ such that for any $k\geq K$ we have that
\begin{align}
    &\frac{1}{2^{2a}} \geq\frac{1}{2\mu(\frac{1}{4}\sigma T^{'})^2}  \big(\frac{(2a)^2M}{k^2} + \frac{(E+\frac{1}{4}L_g Q T^{''})^2}{Mk^{4a}} + \frac{4a(E+\frac{1}{4}L_g Q T^{''})}{k^{2a+1}}\big) +\frac{LQ+CQ+E}{M},\nonumber\\
    \implies &\frac{1}{4}\sigma T^{'}\sqrt{2\mu(T-LQ)}\geq 2aM\frac{1}{k} + (E+\frac{1}{4}L_g Q T^{''}) \frac{1}{k^{2a}}.\nonumber
\end{align}
Inserting this into (\ref{y3}) yields that for sufficiently large $k$,
\begin{equation*}
    \begin{aligned}
        f(X_{k+1}) - f(X_*)&\leq M\frac{1}{k^{2a}} (1-\frac{2a}{k})\leq M(k+1)^{-2a}.
    \end{aligned}
\end{equation*}
Thus, the induction holds.

\medskip
We are now ready to show that $\{\alpha_k\}$ are bounded below by a positive constant for all sufficiently large $k$. From \eqref{oneNSA-G}, we obtain
\begin{align}
    \sigma \alpha_k\|\hat{g}_k\|_{X_k}^2&\leq f(X_k)-f(X_*)-(f(X_{k+1})-f(X_*)) + \gamma_k \nonumber\\
        &\leq f(X_k)-f(X_*)-(f(\dot{X}_{k+1})-f(X_*)) +L\delta_{k+1}+ \gamma_k \nonumber\\
        &\leq f(X_k)-f(X_*)+L\delta_{k+1} +\gamma_k\nonumber\\
         &\leq f(X_k)-f(X_*)+2\gamma_k, \text{ (by Assumption~\ref{assumgamma})}\label{eq16}
\end{align}
     where the second inequality follows from Lemma~\ref{lemmaxx} and  $f$ is Lipschitz continuous with constant $L$ and the third inequality follows from $X_*$ is a local minimizer point. 
     Combing~\eqref{eq16} with~\eqref{p0},~\eqref{w5}, we have that there exist $M'>0$  such that for all sufficiently large $k,$
\begin{equation} \label{eq05}
\alpha_k \|\widehat{g}_k\|_{X_k}^2 \le M' k^{-2a}, \quad \forall k \ge K_1.
\end{equation}

Suppose, by contradiction, that $\liminf_{k\to\infty} \frac{t_k}{\|\widehat{g}_k\|_{X_k}} = 0.$
 Then there exists a subsequence $\mathcal{K}$ such that $\frac{t_k}{\|\widehat{g}_k\|_{X_k}} \to 0$, for $k \in \mathcal{K}$ and $k \to \infty$.
For a sufficiently large $k \in \mathcal{K}$, we have $\alpha_k = \frac{t_k}{2\|\widehat{g}_k\|_{X_k}}$, and hence
\[
\alpha_k = \frac{t_k^2}{4\alpha_k \|\widehat{g}_k\|_{X_k}^2}
\ge \frac{(T')^2}{4M'} > 0,
\]
where the first inequality follows from~\eqref{eq06} and~\eqref{eq05}, which yields a contradiction.

Therefore, the accepted step sizes are eventually bounded away from
zero. There exists a positive constant $\nu$ such that $\{\frac{t_k}{\|\widehat{g}_k\|_{X_k}}\}$ is bounded below $\nu$ for all sufficiently large $k$. By Theorem~\ref{alpha1}, it follows that
$
\alpha_k \ge \min\{\Bar{\alpha}, \frac{1}{2}\nu, \tfrac{2\lambda}{L'}(1-\sigma)\} > 0
$
for all sufficiently large $k$.
\end{proof} 
The local convergence rate of Algorithm~\ref{alg:IRGDoneNS} is established in Theorem~\ref{linear}.
\begin{theorem}\label{linear}
    Let $\{X_k\}$ be the sequence generated by Algorithm \ref{alg:IRGDoneNS}. Suppose that Assumptions \ref{assumgradL},~\ref{assum1} \ref{assumgamma} and \ref{PL}  hold, that all accumulation points of $\{X_k\}$ are local minimizers, that $t_k = \Theta(\frac{1}{k^a})$, for $a \in (\frac{1}{4}, \frac{1}{2})$, and that $0<\sigma<\min\{\frac{1}{\mu\alpha^{'}},1\}$, where $\alpha'$ is defined in Theorem~\ref{conge rate}. Then, there exists $K_1$ such that for all $k\geq K_1$,
    \begin{align}
    f(X_{k+1}) - f(X_*) & \leq \rho^{k-K_1}C_{K_1}, \delta_{k+1}\leq \rho^{k-K_1}\frac{C_{K_1}}{\rho^{'}} \hbox{ and }\nonumber\\
    \dist(X_{k+1},\mathcal{S} ) &\leq \sqrt{\rho}^{k-K_1} (\sqrt{\frac{2}{\mu}(1+\frac{L}{\rho^{'}}) C_{K_1} }+ \frac{1}{\rho^{'}}C_{K_1}),
    \end{align}
    where $X_*$ is an accumulation point of $\{X_k\}$, $\rho,\rho^{'},Q^{'},C^{'}$, and $C_{K_1}$ are positive constants satisfying $\rho \in (1-\mu\sigma\alpha^{'},1), \rho^{'}>C^{'}/\rho,$ $\gamma_k \leq  Q^{'}(\delta_k+(\delta_k+\Bar{\alpha}^2\|\widehat{g}_k\|_{X_k}^2)^2)$, $ C^{'}=
    L+\sigma \Bar{\alpha} L_gM_g +L(1-\mu\sigma\alpha^{'})+Q^{'},$ and $C_{K_1} = f(X_{K_1})-f(X_*) + \rho^{'}\delta_{K_1}.$  
\end{theorem}
\begin{proof}

Combining~\eqref{oneNSA-G},~\eqref{b}, and that $f$ is Lipschitz continuity with a Lipschitz constant $L$ yields 
\begin{align}
    f(X_{k+1}) - f(X_*) \leq& f(\dot{X}_k)-f(X_*) -\sigma \alpha_k\|\widehat{g}_k\|_{X_k}^2 +\gamma_k+L\delta_k\nonumber\\
     \leq& f(\dot{X}_k)-f(X_*) -\frac{1}{2}\sigma \alpha_k\|\widehat{g}_k\|_{F}^2 +\gamma_k+L\delta_{k} \text{ (by~\eqref{eq09})}\nonumber\\
    \leq& (1-\mu\sigma\alpha^{'}) (f(\dot{X}_k)-f(X_*)) +\frac{1}{2}\sigma \Bar{\alpha} L_g\delta_k\|\widehat{g}_k\|_{\F}+\frac{1}{2}\sigma \Bar{\alpha}L_g\delta_k\|\dot{g}_k\|_{\F}\nonumber\\
    &+\gamma_k+L\delta_{k}\text{ (by~\eqref{eq11}, Assumption~\ref{PL}, and Theorem~\ref{conge rate}) }\nonumber\\
    \leq& (1-\mu\sigma\alpha^{'}) (f(X_k)-f(X_*)) +(L+\sigma \Bar{\alpha} L_gM_g+ L(1-\mu\sigma\alpha^{'}))\delta_k\nonumber\\
    &+\gamma_k, \text{ (by $\|\widehat{g}_k\|_F\leq M_g$ and $f(\dot{X}_k)\leq f(X_k)+L\delta_k$) }\label{eq30}
\end{align}
for sufficiently large $k.$
By Assumption~\ref{assumgamma}, there exists $Q^{'}>0$ such that for sufficiently large $k$, $\gamma_k \leq Q^{'}  (\delta_k+(\delta_k+\Bar{\alpha}^2\|\widehat{g}_k\|_{X_k}^2)^2)\leq Q^{'}(\delta_k+(\delta_k+\Bar{\alpha}^2\|\widehat{g}_k\|_{\F}^2)^2).$  Substituting this into~\eqref{eq30} gives
\begin{align}
    f(X_{k+1}) - f(X_*) \leq& (1-\mu\sigma\alpha^{'}) (f(X_k)-f(X_*)) +C^{'}\delta_k+Q^{'}(\delta_k+\Bar{\alpha}^2\|\widehat{g}_k\|_{\F}^2)^2\label{eq13}
\end{align}
for sufficiently large $k.$ To control the quadratic perturbation term in~\eqref{eq13},
we next establish an upper bound for
$\|\widehat g_k\|_F^2$
in terms of the objective residual $f(X_k)-f(X_*)$ and $\delta_k$. Combining $\gamma_k \leq Q^{'}(\delta_k+(\delta_k+\Bar{\alpha}^2\|\widehat{g}_k\|_{\F}^2)^2) $ with~\eqref{eq16} yields 
\begin{align}
    \sigma \alpha^{'}\|\widehat{g}_k\|_{\F}^2&\leq \sigma \alpha_k\|\widehat{g}_k\|_{\F}^2\leq 4\sigma \alpha_k\|\widehat{g}_k\|_{X_k}^2\leq 4(f(X_k)-f(X_*))+8Q^{'}\delta_k+ 8Q^{'}(\delta_k+\Bar{\alpha}^2\|\widehat{g}_k\|_{\F}^2)^2\nonumber\\
    &= 4(f(X_k)-f(X_*))+8Q^{'}\delta_k+8Q^{'}\delta_k^2 +16Q^{'}\delta_k\Bar{\alpha}^2\|\widehat{g}_k\|_{\F}^2+ 8Q^{'}\Bar{\alpha}^4\|\widehat{g}_k\|_{\F}^4\nonumber\\
    &\leq 4(f(X_k)-f(X_*))+16Q^{'}(1 +\Bar{\alpha}^2M_g^2 )\delta_k+ 8Q^{'}\Bar{\alpha}^4\|\widehat{g}_k\|_{\F}^4,\label{eq17}
\end{align}
where the last inequality follows from $\delta_k\leq 0.5\leq 1, \alpha_k\leq\Bar{\alpha},$ and $\|\widehat{g}_k\|_{\F}\leq M_g.$ By Theorem~\ref{IRGDconveroneNS}, we have $\|\widehat{g}_k\|^2_{\F}\leq \frac{\sigma\alpha^{'}}{16Q^{'}\Bar{\alpha}^4}$ for any sufficiently large $k.$ Inserting this into~\eqref{eq17} gives 
\begin{align}
    &\frac{1}{2} \sigma \alpha^{'}\|\widehat{g}_k\|_{\F}^2 \leq 4(f(X_k)-f(X_*))+16Q^{'}(1 +\Bar{\alpha}^2M_g^2 )\delta_k\nonumber\\
    \implies & (\delta_k+\Bar{\alpha}^2\|\widehat{g}_k\|_{\F}^2)^2\leq (\delta_k +\frac{8\Bar{\alpha}^2}{\sigma\alpha^{'}}(f(X_k)-f(X_*)) +\frac{32Q^{'}\Bar{\alpha}^2(1+\Bar{\alpha}^2M_g^2)}{\sigma\alpha^{'}}\delta_k )^2\nonumber\\
     \leq  & C_0 (f(X_k)-f(X_*)+\delta_k)^2\label{eq18}
\end{align}
for sufficiently large $k,$ where $C_0 = (1+\frac{32Q^{'}\Bar{\alpha}^2(1+\Bar{\alpha}^2M_g^2)}{\sigma\alpha^{'}})^2 + (\frac{8\Bar{\alpha}^2}{\sigma\alpha^{'}})^2.$ Substituting~\eqref{eq18} into~\eqref{eq13} yields 
\begin{align}
    f(X_{k+1}) - f(X_*) \leq&(1-\mu\sigma\alpha^{'}) (f(X_k)-f(X_*)) +C^{'}\delta_k+Q^{'}(\delta_k+\Bar{\alpha}^2\|\widehat{g}_k\|_{\F}^2)^2\nonumber\\
    \leq & (1-\mu\sigma\alpha^{'}) (f(X_k)-f(X_*)) +C^{'}\delta_k +Q^{'} C_0 (f(X_k)-f(X_*)+\delta_k)^2\label{eq31}
\end{align}
Since $f(X_k)-f(X_*)\to 0 $ and $ \delta_k \to 0$, we can choose $\rho \in (1-\mu\sigma\alpha^{'},1)$ and $\rho^{'} >C^{'}/\rho $ such that for any sufficiently large $k$,
\begin{align}
    (Q^{'}+\rho^{'}) C_0 (f(X_k)-f(X_*)+\delta_k)^2\leq (\rho -(1-\mu\sigma\alpha^{'}))(f(X_k)-f(X_*)) + (\rho \rho^{'}-C^{'})\delta_k.\label{eq32}
\end{align}
From~\eqref{eq31}, we obtain  
\begin{align}
    f(X_{k+1}) - f(X_*) + \rho^{'}\delta_{k+1} \leq& (1-\mu\sigma\alpha^{'}) (f(X_k)-f(X_*)) + C^{'}\delta_k\nonumber\\
    &+ Q^{'} C_0 (f(X_k)-f(X_*)+\delta_k)^2+\rho^{'}\delta_{k+1}\nonumber\\
    \leq & (1-\mu\sigma\alpha^{'}) (f(X_k)-f(X_*)) + C^{'}\delta_k +  (Q^{'}+\rho^{'}) C_0 (f(X_k)-f(X_*)+\delta_k)^2\nonumber\\
    & \text{ (by $\delta_{k+1}\leq (\delta_k +\Bar{\alpha}^2\|\widehat{g}_k\|_{\F}^2)^2$ and~\eqref{eq18})}\nonumber\\
    \leq &(1-\mu\sigma\alpha^{'}) (f(X_k)-f(X_*)) + C^{'}\delta_k\nonumber\\
    &+(\rho -(1-\mu\sigma\alpha^{'}))(f(X_k)-f(X_*)) + (\rho \rho^{'}-C^{'})\delta_k\text{ (by~\eqref{eq32})}\nonumber\\
     =& \rho (f(X_k)-f(X_*) + \rho^{'}\delta_k ).\nonumber
\end{align}
Thus, there exists $K_1$ such that for all $k\geq K_1$, $f(X_{k+1}) - f(X_*) + \rho^{'}\delta_{k+1}  \leq \rho^{k-K_1}C_{K_1}$. Consequently, 
\begin{align}
    f(X_{k+1}) - f(X_*)  &\leq \rho^{k-K_1}C_{K_1}, 
    \delta_{k+1}\leq \rho^{k-K_1}\frac{C_{K_1}}{\rho^{'}}\label{eq34}
\end{align}
i.e., $\{f(X_k)-f(X_*)\}$ and $\{\delta_k\}$ converges linearly.

To translate the objective-value convergence into iterate
convergence, we invoke the local quadratic growth property. By Assumption~\ref{PL} and~\cite[Proposition~2.2]{QR2025Fast}, we have 
\begin{align}
    f(X)-f^*_{\mathcal{S}} \geq \frac{\mu}{2} \dist (X,\mathcal{S})^2,
\end{align}
for any point $X\in\St(p,n)\cap \mathcal{D}(\mathcal{S},2r).$  
Since we have proved for sufficiently large $k$, $\dist(\dot{X}_k, \mathcal{S})\leq 2r$ in Theorem~\ref{conge rate}, we have 
\begin{align}
   \dist(X_{k+1},\mathcal{S} ) \leq&  \dist (\dot{X}_{k+1},\mathcal{S})+\delta_{k+1} \leq \sqrt{\frac{2}{\mu} (f(\dot{X}_{k+1})-f(X_*))}+\delta_{k+1} \nonumber\\
    \leq& \sqrt{\frac{2}{\mu} (f(X_{k+1})-f(X_*) + L\delta_{k+1}) }+\delta_{k+1}\nonumber \\
    \leq&\sqrt{\rho}^{k-K_1} \sqrt{\frac{2}{\mu}(1+\frac{L}{\rho^{'}}) C_{K_1} } +\rho^{k-K_1}\frac{1}{\rho^{'}}C_{K_1}\text{ (by~\eqref{eq34})}\nonumber\\
    & \leq \sqrt{\rho}^{k-K_1} (\sqrt{\frac{2}{\mu}(1+\frac{L}{\rho^{'}}) C_{K_1} }+ \frac{1}{\rho^{'}}C_{K_1}),
\end{align}
which establishes the R-linear convergence of $\{\dist(X_{k},\mathcal{S})\}$.
\end{proof}
\begin{remark} \label{remark2}
The linear convergence factor obtained in Theorem~\ref{linear}
is
\[
\rho \in (1-\mu\sigma\alpha',\,1).
\]
Hence the asymptotic convergence speed is governed by the local
P$\text{\L}$ constant $\mu$ and uniform positive lower bound $\alpha^{'}$ on the final accepted step size. When $f$ is twice continuously differentiable around a nondegenerate local minimizer $X^\star$, i.e., $\mathrm{Hess}f(X^\star)\succ0$, the local P$\text{\L}$ constant satisfies $\mu  = \Theta (\lambda_{\min}\!\left(
\mathrm{Hess}\,f(X^\star)
\right))$, where $\lambda_{\min}
(\mathrm{Hess}\,f(X^\star))$ denotes the smallest eigenvalue of $\mathrm{Hess}\,f(X^\star)$.
On the other hand, since $\{t_k\}$ converges to zero sublinear and $\{\|\widehat{g}_k\|_{X_k}\}$ converges to zero linearly, then 
\begin{align}
    \frac{t_k}{\|\widehat{g}_k\|}_{X_k}\to +\infty.\label{tk/gk}
\end{align} 
It follows from the proof of Theorem~\ref{conge rate} that $\nu = +\infty$ and $\alpha^{'}\leq \min\{\Bar{\alpha},\frac{2\lambda}{L^{'}}(1-\sigma)
\}=O
\!\left(
\frac{1}{\lambda_{\max}
(\mathrm{Hess}\,f(X^\star))}
\right)$, where $\lambda_{\max}
(\mathrm{Hess}\,f(X^\star))$ denotes the largest eigenvalue of $\mathrm{Hess}\,f(X^\star)$.
Consequently,
$
1-\mu\sigma\alpha'
=
1-O
\!\left(
\frac{\lambda_{\min}
(\mathrm{Hess}\,f(X^\star))}
{\lambda_{\max}
(\mathrm{Hess}\,f(X^\star))}
\right)
=
1-O
\!\left(
\kappa^{-1}
\right),
$
where $\kappa
=
\frac{\lambda_{\max}
(\mathrm{Hess}\,f(X^\star))}
{\lambda_{\min}
(\mathrm{Hess}\,f(X^\star))} $
denotes the condition number of the Riemannian Hessian at $X^{\star}$. 
Therefore, the parameter $\rho$ of the local linear convergence rate is related to the function itself, the infeasibility in the proposed algorithm
does not influence the local convergence. Moreover, the local convergence rate deteriorates as the Hessian
becomes increasingly ill-conditioned, which is consistent with the behaviour of the exact Riemannian gradient descent method. The proposed inexact algorithm preserves the same local
linear convergence rate as exact Riemannian gradient descent while
significantly reducing the computational cost of retraction.
\end{remark}


\begin{remark}\label{remark3}
The update rule of Landing is given as follows:
\begin{align}
    X_{k+1} = X_{k}-\alpha (\widehat{g}(X_k) + \beta X_k(X_k^TX_k-I).
\end{align}
where $\beta$ is penalty paramater and related to the gradient of the penalty function $\frac{1}{4}\|X_k^TX_k-I\|_{\F}^2$. Therefore, it require sufficiently large penalty parameters to enforce the feasibility of the iterate. Moreover,
the local convergence factor~\cite[Theorem 1]{SYB2024Local} is 
\begin{align}
   1-\frac{\alpha \rho\mu^{'}}{2}, \rho=
\min
\left\{
\frac{1}{2},
O(\frac{1}{\beta})
\right\}, \frac{1}{\mu'} = \max \left\{ \frac{1}{\mu}, \frac{B_1}{\beta^2} \right\}
\nonumber
\end{align}
where $\mu$ is  Riemannian local P$\text{\L}$ constant, $B_1$ are constants associated with the function $f.$ 
For a small $\beta$, $\mu'=O(\beta^2)$, indicating insufficient constraint enforcement and a slow convergence rate. In contrast, when $\beta$ becomes excessively large, $\rho = O(1/\beta)$, which also deteriorates the convergence factor. Therefore, the choice of $\beta$ requires a delicate balance between feasibility improvement and convergence speed. 
ExPen solves the original problem by solving the exact penalty function $h$, which is constructed as follows:
\begin{align}
    h(X) := f\left(X\left(\frac{3}{2}I_p-\frac{1}{2}X^\top X\right)\right)
+\frac{\beta}{4}\left\|X^\top X-I_p\right\|_F^2.
\end{align}
 A small penalty parameter $\beta$ may lead to the failure of convergence, while a large penalty parameter may result in a large condition number of  penalty function, thus lead to slow convergence rate~\cite{xiao2024solving}. PLAM is updated according to the following rules:
\begin{align}
   X_{k+1} =  X_k-\alpha_k
\left(
\nabla f(X_k)
-X_k\mathrm{sym}\left(\nabla f(X_k)^\top X_k\right)
+\beta X_k\left(X_k^\top X_k-I_p\right)
\right)
\end{align} and PCAL is a modification of PLAM. They rely on sufficiently large penalty parameter $\beta$ to establish global convergence guarantees and the local convergence factor is $1-\frac{\lambda_{\min}(\Hess f (X^*))}{B_2+2\beta}$~\cite[Theorem 4.2]{BinGao2019paraOrth}, where $B_2$ is constants associated with the function $f$. Increasing $\beta$ enlarges the convergence factor and thus slows down the local convergence. 

Therefore, existing infeasible methods generally suffer from a trade-off between feasibility enforcement and local convergence speed due to the dependence of their convergence behavior on penalty parameters.

Since IRGD-StieONS does not use any penalty term, tuning the coefficient of the penalty is no longer required. The parameter sequence $\{\gamma_k\}$ used in IRGD-StieONS may seem difficult to tune (see Assumption~\ref{assumgamma}), yet this is not the case in practice. Firstly, a feasible scheme is given in~\eqref{gammakk}, indicating that Assumption~\ref{assumgamma} can be satisfied. Secondly, the numerical experiment in Section~\ref{subsec:5.2} shows that the performance of IRGD-StieONS is insensitive to $\{\gamma_k\}$.
\end{remark}

\begin{remark}\label{stepsize}
Unlike many existing algorithms, IRGD-StieONS exhibits great flexibility for selecting initial step sizes.
For Landing, the analysis of its local linear convergence rate does not support dynamic step size and requires
  $\alpha\leq O(\frac{1}{\beta})$ .
In PLAM and PCAL~\cite{BinGao2019paraOrth}, dynamic step size is allowed, but ensuring global convergence and local linear convergence rate requires that the step size be sufficiently small, i.e.,  $\alpha_k\leq  O(\frac{1}{\beta} ).$ 
In contrast, IRGD-StieONS does not introduce a penalty parameter, and hence its stepsize selection is not subject to a penalty-parameter-dependent upper bound. When $k$ is sufficiently large, $\frac{t_k}{2\|\widehat{g}_k\|_{x_k}}\to \infty$(from~\eqref{tk/gk}), the restriction on the initial step size weakens. 
Note that $\Bar{\alpha}$ is not necessarily a constant. It can be a bounded sequence depending on $k$, i.e., $\{\Bar{\alpha}_k\}$ satisfying $0 < \alpha_{-} \leq \Bar{\alpha}_k \leq \alpha_+$ for all $k$, where $\alpha_{-}$ and $\alpha_{+}$ are two constants. It follows that all theoretical results still hold.
Therefore, a practical option for $\{\Bar{\alpha}_k\}$ is $\Bar{\alpha}_k = \min(\max(\alpha_{-}, \alpha^{\mathbb{BB}}_k), \alpha_+)$, where $\alpha^{\mathbb{BB}}_k$ denotes a BB step size.
    
\end{remark}
\section{An Inexact Riemannian Stochastic Gradient Descent with one Newton-Schulz Iteration}\label{Sec4}

The proposed IRGD-StieONS can be adjusted to solve optimization problems in the form of
\begin{equation}
    \begin{aligned}
        \min_{X\in\St(p.n)} f(X) =  \frac{1}{N}\sum_{i=1}^N f_i(X),
    \end{aligned}
\end{equation}
where each function $f_i$ is continuously differentiable.
At each iteration, it takes a random function $f_i$ and goes in
the direction opposite to its Riemannian gradient. The index $i$ is drawn from the discrete uniform distribution over $\{1,\dots,N \}$, then the stochastic gradient is an unbiased estimator of $\widehat{g}$, i.e., $\mathbb{E}_i[\widehat{g}^{i}(X)] = \widehat{g}(X),$ where $\widehat{g}^{i}(X) = \mathrm{skew}(\nabla f_i(X)X^T)X$ and $\widehat{g}(X) = \mathrm{skew}(\nabla f(X)X^T)X$. The resulting algorithm, called an inexact Riemannian stochastic gradient descent with one Newton-Schulz iteration (IRSGD-StieONS), is stated in Algorithm~\ref{alg:RSGDoneNS}. Note that Algorithm~\ref{alg:RSGDoneNS} is a variant of the Riemannian stochastic gradient algorithm in~\cite{bonnabel2013stochastic} on the Stiefel manifold, which additionally uses a line search for step size selection and replaces the exponential map with one Newton-Schulz iteration.
\begin{algorithm}[!ht]
    \caption{An inexact Riemannian Stochastic Gradient Descent on the Stiefel manifold with one Newton-Schulz iteration (IRSGD-StieONS) }
    \label{alg:RSGDoneNS}
\renewcommand{\algorithmicrequire}{\textbf{Input:}}
    \renewcommand{\algorithmicensure}{\textbf{Output:}}
    \begin{algorithmic}[1]
        \REQUIRE Initial point $X_0 \in \St(p,n)^{0.5}$; a sequence $\{t_k\}$ and a scalar $\Bar{\alpha}>0$ for initial step size selection; and a sequence $\{\gamma_k\}$ and two scalars $\sigma \in (0, 1)$ and $\lambda \in (0,1)$ for line search conditions.  
        \FOR {$k = 0, 1, 2, \ldots$}
        \STATE 
        $\widehat{g}^{i_k}_k =  \mathrm{skew}(\nabla f_{i_k}(X_k)X_k^T)X_k,i_k\sim\mathcal{U}[1,N]$ and $ \widehat{d}^{i_k}_k= - \widehat{g}^{i_k}_k$;
        \STATE $\alpha_k = \Bar{\alpha}_k =  \min(\Bar{\alpha},\frac{t_k}{2\|\widehat{g}^{i_k}_k\|_{X_k}})$;
         \WHILE{ 
         \label{alg:step4}$f_{i_k}(\Psi(X_k+\alpha_k\widehat{d}^{i_k}_k)) > f_{i_k}(X_k) + \sigma \alpha_k \left\langle\hat{g}^{i_k}_k,\hat{d}^{i_k}_k\right\rangle_{X_k} + \gamma_k$} 
          \STATE   $\alpha_k = \lambda \alpha_k$; 
         \ENDWHILE 
        \STATE Set $X_{k+1}=\Psi(X_k+\alpha_k\widehat{d}^{i_k}_k);$ 
        \ENDFOR
    \end{algorithmic}
\end{algorithm}


Next, we establish its global convergence, which relies on Assumption~\ref{RSGD:fL}.
\begin{Assumption}\label{RSGD:fL}
     The gradient of the function $f_i$ is Lipschitz continuous with Lipschitz constant $\Bar{L}_i$ on $\St(p,n)^{0.5}$, where $i = 1,2,\dots N.$  
\end{Assumption}
Some constants are defined as follows.
(i) It follows from Assumption~\ref{RSGD:fL} that the extended Riemannian gradient $\widehat{g}^{i}(X) = \mathrm{skew}(\nabla f_i(X)X^T)X$ is Lipschitz continuous in $\St(p,n)^{0.5},$ with Lipschitz constant denoted by $L^i_g.$ (ii) Since the function
$f_i$ is continuously differentiable, it is Lipschitz continuous in the compact set $\St(p,n)$, with Lipschitz constant denoted by $L_i$. (iii) 
Since $\widehat{g}^i$ is a continuous function, there exists an upper bound constant of $\|\widehat{g}^i\|_{\F}$ on the compact set $\St(p, n)^{0.5}$, i.e., 
\begin{align}
    M_g^{i}: = \sup_{X\in \St(p,n)^{0.5}}\|\widehat{g}^i(X)\|_{\F}\nonumber
\end{align}
is finite. (iv) By Assumption~\ref{RSGD:fL} and~\cite[Lemma 2.7]{NB2018Global}, there exists constant $L^{'i}$ such that $f_i$ is $L^{'i}$-retraction-smooth with respect to the retraction $R$ in $\St(p,n),$ i.e.,
\begin{equation*}
    f_i(R_{X}(V))\leq f_i(X) + \left\langle\grad f(X), V \right\rangle_{X}+\frac{L^{'}_i}{2}\|V\|_X^2, \forall X\in\St(p,n), V\in \T_X\St(p,n).
\end{equation*}
We further require that $L \geq  \max_{i = 1,2,...,N}\{L_i\}, \Bar{L}\geq \max_{i = 1,2,...,N}\{\Bar{L}_i\}$, $L_g \geq \max_{i = 1,2,...,N}\{L_g^i\}$, $M_g \geq \max_{i = 1,2,...,N}\{M_g^i\}$ and $L^{'} \geq \max_{i = 1,2,...,N}\{L^{'}_i\}$. The sequences $\{\theta_k\}$, $\{\delta_k\}$, and $\{\zeta_k\}$ are defined as those in Section~\ref{Sec3}, and Lemma~\ref{B117} remains valid for Algorithm~\ref{alg:RSGDoneNS}.
Denote $\dot{X}_k = \P_{\St(p,n)}(X_{k-1}+\alpha_{k-1}\widehat{d}_k^{i_k}), \dot{g}_k^{i_k} =\mathrm{skew}(\nabla f_{i_k}(\dot{X}_k)\dot{X}^T_k)\dot{X}_k$, and $\dot{d}_k^{i_k} = -\dot{g}_k^{i_k}.$ Based on Assumptions~\ref{assum1},~\ref{assumgamma}, and~\ref{RSGD:fL}, following the proofs of Lemma~\ref{lemmaxx} analogously yields
\begin{align}
    \|X_k-\dot{X}_k\|_{\F}&\leq \delta_k;\label{i3}\\
    \|\widehat{g}_k^{i_k}- \dot{g}_k^{i_k}\|_{\F}&\leq L_g\delta_k;\label{i2}\\
    \|\Psi (X_k+\alpha_k\hat{d}_k^{i_k} )-R_{\dot{X}_k}(\alpha_k\dot{d}^{i_k}_{k})\|_{\F}&\leq \zeta_{k+1}.\label{i1}
 \end{align} 
\begin{theorem}\label{salpha1}
    Suppose Assumptions~\ref{assumgradL}, \ref{assum1}, \ref{assumgamma}, and~\ref{RSGD:fL} hold. 
    Let $\alpha_k$ be the accepted step size in the $k$th
iteration of Algorithm~\ref{alg:RSGDoneNS}. Then we have
\begin{equation*}
    \alpha_k\geq \min\{\Bar{\alpha}_k,\frac{2\lambda}{L^{'}} (1-\sigma)\}  = \min\{\Bar{\alpha},\frac{t_k}{2\|\widehat{g}^{i_k}_k\|_{X_k}}, \frac{2\lambda}{L^{'}} (1-\sigma)\} \quad \forall k, \text{ and }\sum_{k = 0}^{\infty}\alpha_k = \infty.
\end{equation*}
The proof of Theorem~\ref{salpha1} follows the same idea as Theorem~\ref{alpha1} by simply replacing $f$ with $f_{i_k}$.
\end{theorem}
Theorem~\ref{RSGD:conver} shows that our method has the same convergence rate as
Riemannian stochastic gradient descent with decreasing step size.
\begin{theorem}\label{RSGD:conver}
    Suppose that Assumptions~\ref{assum1},~\ref{assumgamma}, and~\ref{RSGD:fL} hold. Let $\{X_k\}$ be a sequence generated by Algorithm~\ref{alg:RSGDoneNS}. Assume further that the sequence $t_k =\hat{t}(1+k)^{-\frac{1}{2}}$, where $\hat{t}\leq \min\{2M_g\Bar{\alpha}, \frac{4\lambda(1-\sigma)M_g}{L^{'}},\sqrt{0.2}\}$. Then, for any $K\geq 1$, the following bound holds:
    $$ \min_{k<K}\mathbb{E}[\|\widehat{g}(X_k)\|_{\F}^2]\leq \frac{2M_g}{\hat{t}\sqrt{K}} (f(X_0)-f^*+ \frac{L^{'}\hat{t}}{2}\log(K+1) + (S^{'}+\frac{\Bar{\alpha}}{2}M_gL_g ) \sum_{k=0}^K\theta_k),$$
    where $S^{'}$ is a constant such that $\frac{L^{'}\alpha_k^2}{2}( L_g\delta_k\|\dot{g}^{i_k}_{k}\|_{\F} + L_g\delta_k\|\widehat{g}_k^{i_k}\|_{\F} )+L(\delta_k+ \zeta_{k+1})\leq S^{'} \theta_k$.
    The expectation here is taken with respect to all the random realizations of the random
variables $i_k$, $k\leq K$.
\end{theorem}
\begin{remark} (The existence of $S^{'}$)
By Assumption~\ref{assum1} and  Lemma~\ref{B117}, we have $\theta_{k+1} = (\theta_k+t_k^2)^2 \leq \frac{1}{2}\theta_k + 0.4\lambda\theta_k +t_k^4 .$ Since $\theta_k  = \Theta (t_k^4)$, there exist a constant $S^{''}$ such that $t_k^4\leq S^{''}\theta_k$. It follows that $\theta_{k+1} \leq (\frac{1}{2} + 0.4 \lambda + S^{''}) \theta_k. $
    Combine this with $\alpha_k\leq \Bar{\alpha
    }, \|\dot{g}^{i_k}_{k}\|_{\F}\leq M_g ,  \|\widehat{g}_k^{i_k}\|_{\F} \leq M_g, \delta_k\leq \theta_k$ and $\zeta_{k+1} \leq \theta_{k+1} + L_p(1+\Bar{\alpha}L_g)\delta_k$, we have 
    $\frac{L^{'}\alpha_k^2}{2}( L_g\delta_k\|\dot{g}^{i_k}_{k}\|_{\F} + L_g\delta_k\|\widehat{g}_k^{i_k}\|_{\F} )+L(\delta_k+ \zeta_{k+1}) \leq (L^{'}\Bar{\alpha}^2L_gM_g+ L + LL_p(1+\Bar{\alpha}L_g))\theta_k+L \theta_{k+1} \leq (L^{'}\Bar{\alpha}^2L_gM_g+ L + LL_p(1+\Bar{\alpha}L_g) + L (\frac{1}{2} + 0.4 \lambda + S^{''}) )\theta_k$. Therefore, choosing $S^{'} = L^{'}\Bar{\alpha}^2L_gM_g+ L + LL_p(1+\Bar{\alpha}L_g) + L (\frac{1}{2} + 0.4 \lambda + S^{''})$ yields $\frac{L^{'}\alpha_k^2}{2}( L_g\delta_k\|\dot{g}^{i_k}_{k}\|_{\F} + L_g\delta_k\|\widehat{g}_k^{i_k}\|_{\F} )+L(\delta_k+ \zeta_{k+1})\leq S^{'} \theta_k$.
\end{remark}
\begin{proof}
     By~\eqref{Lrgd}, we have
    \begin{equation*}\label{s4}
        \begin{aligned}
            f(R_{\dot{X}_k}(\alpha_k\dot{d}^{i_k}_{k}))&\leq 
            f(\dot{X}_k) +\alpha_k\left\langle \widehat{g}(\dot{X}_k),\dot{d}_k^{i_k}\right\rangle_{\dot{X}_k} + \frac{\alpha_k^2}{2}L^{'}\|\dot{d}^{i_k}_{k}\|_{\dot{X}_k}^2.
        \end{aligned}
    \end{equation*}
     Combining this with (\ref{i3}),~\eqref{i1},~\eqref{w3} and using $\dot{d}^{i_k}_{k} = -\dot{g}^{i_k}_{k}$ yields
    \begin{equation*}
        \begin{aligned}
            f(X_{k+1})&\leq 
            f(X_k) -\alpha_k\left\langle \widehat{g}(\dot{X}_k),\dot{g}^{i_k}_{k}\right\rangle_{\dot{X}_k} + \frac{\alpha_k^2}{2}L^{'}(\|\hat{g}^{i_k}_{k}\|_{\F}^2 + L_g\delta_k\|\dot{g}^{i_k}_{k}\|_{\F} +L_g\delta_k\|\widehat{g}_k^{i_k}\|_{\F} )+L(\delta_k+ \zeta_{k+1})\\
            &\leq 
            f(X_k) -\alpha_k\left\langle \widehat{g}(\dot{X}_k),\dot{g}^{i_k}_{k}\right\rangle_{\dot{X}_k} + \frac{L^{'}t_k^2}{2}+C^{'}\theta_k,
        \end{aligned}
    \end{equation*}
    where the second inequality follows from $\alpha_k\leq \frac{t_k}{2\|\widehat{g}^{i_k}_{k}\|_{X_k}}$ and the existence of a constant $S^{'}$ such that $\frac{L^{'}\alpha_k^2}{2}( L_g\delta_k\|\dot{g}^{i_k}_{k}\|_{\F} + L_g\delta_k\|\widehat{g}_k^{i_k}\|_{\F} )+L(\delta_k+ \zeta_{k+1})\leq S^{'} \theta_k$. 
    Taking expectations with respect to the random variable $i_k$ yields  
        \begin{align}
            \mathbb{E}_{i_k}[f(X_{k+1})]&\leq 
            f(X_k) -\alpha_k\|\widehat{g}(\dot{X}_k)\|^2_{\dot{X}_k}+\frac{L^{'}t_k^2}{2} + S^{'}\theta_k\nonumber\\
            &\leq 
            f(X_k) -\frac{\alpha_k}{4}\|\widehat{g}(\dot{X}_k)\|^2_{\F}+\frac{L^{'}t_k^2}{2} + S^{'}\theta_k,\label{s5}
        \end{align}
       where we use the stochastic gradient is an unbiased estimator of $\widehat{g}$.
      
      Since $\|\widehat{g}(\dot{X}_k) - \widehat{g}(X_k)\|_{\F}\leq L_g\delta_k$, we have $\|\widehat{g}(\dot{X}_k)\|_{\F}\geq \|\widehat{g}(X_k)\|_{\F}-L_g\delta_k$. Hence, $\|\widehat{g}(\dot{X}_k)\|_{\F}^2\geq \|\widehat{g}(X_k)\|_{\F}^2 -Lg\delta_k(\|\widehat{g}(X_k)\|_{\F}+\|\widehat{g}(\dot{X}_k)\|_{\F}) \geq \|\widehat{g}(X_k)\|_{\F}^2 - 2M_gL_g\theta_k$, 
      where the second inequality follows from $\delta_k\leq \theta_k$ and $\|\widehat{g}(X_k)\|_{\F}+\|\widehat{g}(\dot{X}_k)\|_{\F}\leq 2M_g.$
      Substituting this into~(\ref{s5}) we obtain
      \begin{equation}\label{s6}
          \mathbb{E}_{i_k}[f(X_{k+1})]\leq 
            f(X_k) -\frac{\alpha_k}{4} \|\widehat{g}(X_k)\|_{\F}^2 +\frac{L^{'}t_k^2}{2} + S^{'}\theta_k + \frac{\Bar{\alpha}}{2}M_gL_g \theta_k .
      \end{equation}
       Taking expectations with respect to the past, and summing inequality~(\ref{s6}) from $k=0$ to $K$ gives
      \begin{equation*}\label{s7}
          \sum_{k=0}^K\alpha_k\mathbb{E}[\|\widehat{g}(X_k)\|_{\F}^2 ] \leq f(X_0)-f^*+ \frac{L^{'}}{2}\sum_{k=0}^Kt_k^2 + (S^{'}+\frac{\Bar{\alpha}}{2}M_gL_g ) \sum_{k=0}^K\theta_k.
      \end{equation*}
      From $t_k = \hat{t}(1+k)^{-\frac{1}{2}},$ we have $\alpha_k\geq \min\{\Bar{\alpha},\frac{t_k}{2\|\widehat{g}^{i_k}_{k}\|_{{X}_k}}, \frac{2\lambda}{L^{'}} (1-\sigma)\} \geq \min\{\Bar{\alpha},\frac{t_k}{2M_g}, \frac{2\lambda}{L^{'}} (1-\sigma)\}=\frac{t_k}{2M_g} $.  
      Using the bound 
      $$\min_{k\leq K}\mathbb{E}[\|\widehat{g}(X_k)\|_{\F}^2]\leq \sum_{k=0}^K\alpha_k\mathbb{E}[\|\widehat{g}(X_k)\|_{\F}^2 ]\times (\sum_{k=0}^K\alpha_k)^{-1},$$ 
      together with $\sum_{k = 0}^K(1+k)^{-1}\leq \log(K+1)$ and $\sum_{k=0}^K t_k\geq \hat{t}\sqrt{K}$, we obtain
      \begin{equation*}
          \begin{aligned}
              \min_{k\leq K}\mathbb{E}[\|\widehat{g}(X_k)\|_{\F}^2]&\leq (f(X_0)-f^*+ \frac{L^{'}\hat{t}}{2}\log(K+1) + (S^{'}+\frac{\Bar{\alpha}}{2}M_gL_g ) \sum_{k=0}^K\theta_k) \times (\sum_{k=0}^K\alpha_k)^{-1}\\
              &\leq\frac{2M_g}{\hat{t}\sqrt{K}} (f(X_0)-f^*+ \frac{L^{'}\hat{t}}{2}\log(K+1) + (S^{'}+\frac{\Bar{\alpha}}{2}M_gL_g ) \sum_{k=0}^K\theta_k),
          \end{aligned}
      \end{equation*}
where $\alpha_k \geq \frac{t_k}{2 M_g}$ is used.
\end{proof}

 \section{Numerical experiments}\label{Sec5}
 This section evaluate the numerical performance of algorithms IRGD-StieONS and IRSGD-StieONS. We first introduce the parameter setting and testing environment in Section~\ref{subsec:5.1}. 
Subsequently, we report the numerical results in Sections~\ref{subsec:5.2} and~\ref{subsec:5.3}, respectively. 
The discretized Kohn–Sham total energy minimization and principal component analysis experiments are conducted on a workstation equipped with an Intel Core Ultra 7 155H CPU and 32 GB RAM.
The CNN experiments are conducted on a GPU server equipped with an NVIDIA RTX 4090 GPU and 24 GB GPU memory.
All competing algorithms within each experiment are implemented and evaluated on the same platform.

\subsection{Parameter settings and testing environment}\label{subsec:5.1}
We compare the proposed algorithms with the landing algorithm (Landing)~\cite{ablin2024infeasible, SYB2024Local}\footnote{Landing: \href{https://github.com/simonvary/landing-stiefel}{ https://github.com/simonvary/landing-stiefel}.}, the Riemannian gradient descent algorithm (RGD)~\cite{boumal2023intromanifolds} from the package Geoopt\footnote{Geoopt: \href{https://github.com/geoopt/geoopt}{ https://github.com/geoopt/geoopt}.}, the ExPen algorithm with gradient descent (ExPen-GD)~\cite{xiao2024solving}, and proximal linearized augmented
Lagrangian algorithm (PLAM)~\cite{BinGao2019paraOrth}.
The stochastic version IRSGD-StieONS is compared to the stochastic versions of Landing (Landing-SG), RGD (RSGD), ExPen-GD (ExPen-SGD), and PLAM (PLAM-SG).
We test these algorithms on Discretized Kohn–Sham total energy minimization (DKS)~\cite{liu2014convergence},  principal component analysis (PCA), and convolutional neural networks (CNN). Unless specified otherwise, the proposed algorithms IRGD-StieONS and IRSGD-StieONS set sequence $t_k = \min\{ \frac{2\|\widehat{g}_k\|_{X_k}\sqrt{\sqrt{0.5}-\|X_k^TX_k-I_p\|_{\F}}}{\|\widehat{g}_k\|_{\F}},  \frac{10 }{ (k+1)^{0.3}}\}$ (see Remark~\ref{tk} for details), 
$\gamma_k = 10\beta (\|X_k^TX_k-I\|_{\F}+\|\widehat{g}_k\|_{X_k}^4),$ and
$\theta_{k+1}= 0.5$ if $k< \hat{K}$ and $\theta_{k+1} = (\theta_k+t_k^2)^2$ otherwise,
where $\hat{K}$ denotes the smallest integer such that $\frac{4 }{ (\hat{K}+1)^{0.3}}\leq \lambda\sqrt{0.2}$ holds.
The parameters $\beta$, $\sigma$, and $\rho$ are set to be 1, 0.5, and 0.5, respectively. Note that the sequence $\{\gamma_k\}$ is uniquely defined by $\beta$. 
The RGD and RSGD is initialized at a randomly generated point in $\St(p,n)$, while all other algorithms are initialized at randomly generated points in $\St(p,n)^{0.5}$. All the experiments are implemented in Python using PyTorch.
The DKS and PCA experiments are conducted on a CPU. The corresponding algorithms stop if the absolute error $|f(X_{k})-f(X_{*})|\leq 10^{-12}$ and the feasibility $\|X_k^TX_k-I\|_{\F}\leq 10^{-14}$, where $X_*$ is a highly accurate approximation of the solution and will be specified later. The CNN experiments are conducted using a fixed number of epochs and run on a single GPU.\\
\textbf{DKS:} 
Consider a simplified discretized Kohn--Sham total energy
minimization problem~\cite{liu2014convergence}:
\begin{equation}
\begin{aligned}
    \min_{X\in\mathbb{R}^{n\times p}}\quad
    &\frac{1}{2}\operatorname{tr}(X^\top L X)
    +\frac{\alpha}{4}\rho(X)^\top L^\dagger \rho(X),\\
    \text{s.t.}\quad
    &X^\top X=I_p,
\end{aligned}
\end{equation}
where $L\in\mathbb{S}^n : = \{X\in \mathbb{R}^{n\times n} \mid X^T = X\}$,
$\rho(X)=\operatorname{diag}(XX^\top)$, and $L^{\dagger} : =(L^TL)^{-1}L^T $ refers to the pseudoinverse of $L$. We set $\alpha=1$ in all
experiments. The matrix $L$ is set as $
L=\frac{1}{2}\left(\widetilde{L}+\widetilde{L}^{\top}\right)$, where each entry of $\widetilde{L}$ is drawn from the standard normal distribution.
\\
\textbf{Parameter setting for DKS:} We set $n=1000$ and $p=50$. 
The penalty parameter $\beta$ is set to $10$, $10$, and $10\|L\|_2$ for Landing, ExPen-GD, and PLAM, respectively, and the safe-region parameter for Landing is set to $\frac{1}{2}$, since they work efficiently and reliably in our experiments (See Section~\ref{subsec:5.2}). The initial step sizes at the first iterations are set to $0.1$, $0.01$, $0.02$, and $0.02$ for Landing, RGD, ExPen-GD, and PLAM, respectively. For the first iteration of IRGD-StieONS, we set  $\Bar{\alpha} = 0.05$.
For a fair comparison, all tested algorithms employ the same
alternating Barzilai--Borwein (BB) stepsize strategy~\cite{dai2005projected}, in which the
BB1 and BB2 stepsizes are used alternately. Throughout this section, the stepsize selection strategy for IRGD-StieONS and IRSGD-StieONS refers to the selection of $\Bar{\alpha}$.  Note that although the alternating BB stepsizes seem to work well in practice for all algorithms, they are not guaranteed to satisfy the stepsize assumption in Landing and PLAM.

Since the global minimizers of the considered nonconvex problems are 
generally unavailable, we use the Riemannian trust-region (RTR) solver~\cite{townsend2016pymanopt}  
provided by the Pymanopt toolbox to obtain a high-accuracy reference 
solution. The RTR iterations 
are terminated when
$
    \|\operatorname{grad} f(X_k)\|_{\mathrm{F}}\leq 10^{-12}.
$
The obtained solution, denoted by $X_{*}$, is used as a numerical 
approximation of the minimizer.\\
\textbf{PCA:}
Consider the following PCA problem~\cite{hotelling1933analysis}:
\begin{align} \min_{X\in\St (p,n)}
-\frac{1}{2}\operatorname{tr}(X^{T}AX),\nonumber
\end{align}
where $A=\frac{1}{m} Y^{T}Y$ denotes the sample covariance matrix of the data matrix
$Y\in\mathbb{R}^{m\times n}$. Specifically, we generate
$Y=U\Sigma V^{T}$, where $U$ and $V$ are random orthonormal matrices
obtained by QR factorizations. The first $p$ diagonal entries of $\Sigma$
are chosen to decrease linearly from $10$ to $0.5$, i.e.,
$
\Sigma_{ii}
=
10-\frac{9.5(i-1)}{p-1},  i=1,\ldots,p,$
while the remaining singular values are set to zero. The generation of problem instances is inspired by~\cite{gao2018new}. 
The optimal solution $X_*$ is given by the $p$ leading
eigenvectors of $A$. For stochastic optimization, the objective admits the finite-sum
representation $
f(X)
=
-\frac{1}{2m}\sum_{i=1}^{m}\|y_i^{T}X\|_2^2,$
with the mini-batch stochastic gradient $
\nabla f_{\mathcal{B}}(X)
=
-\frac{1}{|\mathcal{B}|}
Y_{\mathcal{B}}^{T}Y_{\mathcal{B}}X,$
where $\mathcal{B}$ is a randomly sampled mini-batch. This formulation
provides a standard finite-sum PCA problem for testing stochastic
Riemannian optimization methods.\\
\textbf{Parameter setting for PCA:} We test both deterministic and stochastic cases.  For the deterministic setting, we set
$n=1000$, $m=500$, and $p=100$. For the stochastic setting, we consider
a larger-scale problem with $n=10000$, $m=1000$, and $p=100$, and use a
mini-batch of size $500$ to compute the stochastic gradient. 

We set $\beta=1$ for all penalty-based methods, except for ExPen-GD, for which $\beta=30$, and use a safe-region parameter of $\frac{1}{2}$ for Landing and Landing-SG. The initial step sizes at the first iterations for Landing, Landing-SG, RGD, RSGD, ExPen-GD, ExPen-SGD, PLAM, and PLAM-SG are set to $0.01$, $0.1$, $1$,$0.01$, $0.1$, $0.01$, $10$, and $50$, respectively. For the first iteration of IRGD-StieONS and IRSGD-StieONS, we set  $\Bar{\alpha} = 0.1$.
The alternating BB stepsize strategy is used
for the deterministic setting, and a prescribed stepsize strategy is used for the stochastic
setting. For the prescribed step size setting, we mimic the step size decay strategy used in~\cite{ablin2024infeasible}. Specifically, the step size is multiplied by a decay factor of $0.9$ at the $30$th and $60$th epochs, respectively. \\
\textbf{CNN:} We test standard model VGG16~\cite{simonyan2014very} for image classification on dataset CIFAR10~\cite{krizhevsky2009learning}. In VGG16, for a convolutional layer with kernel 
$\widehat{K} \in \mathbb{R}^{c_{out}\times c_{in}\times h\times w}$,
we first reshape $\widehat{K}$ into a matrix $K$ of size $p\times n$,
where $p=c_{out}$, $n=c_{in}\times h\times w$.
In the case when the reshaping results in a wide instead of a tall matrix, we impose the orthogonality on its transposition.
Then, we restrict the matrix $K$ on the Stiefel manifold using Landing-SG, RSGD, ExPen-SGD, PLAM-SG, and IRSGDStieONS, while other parameters are optimized with SGD. The CIFAR-10 dataset consists of 60000 color images from 10 classes and officially divided into 50000 training images and 10000 testing images, with 5000 training samples and 1000 testing samples per class.\\
\textbf{Parameter setting for CNN:} For all of the tested algorithms, the total number of epochs in training is 50, the batch size is 100, and the initial learning rates are set as $0.01$ for PLAM-SG and $0.11$ for others. During training, we reduce the learning rates by a factor of $0.1$ at 10, 20 epochs. Parameter $\Bar{\alpha}$ in IRSGD-StieONS is chosen as learning rate. For the Landing-SG, ExPen-SGD and PCAL-SG, the penalty parameter to be 1. The safe
region parameter for Landing-SG to be $\frac{1}{2}.$ 


 \begin{remark}\label{tk}
     In practical computations, we adopt the following adaptive step-size rule to reconcile numerical efficiency with the theoretical convergence guarantees:
\begin{equation}
t_k = \min\left\{ 
\frac{2\,\|\widehat{g}_k\|_{X_k}\,\sqrt{\sqrt{0.5}-\|X_k^\top X_k - I_p\|_{\mathrm F}}}{\|\widehat{g}_k\|_{\mathrm F}},\; 
\frac{10}{(k+1)^{0.3}}
\right\},
\label{eq:stepsize_rule}
\end{equation}
The first term is designed to control the magnitude of the update according to the current distance of $X_k$ from the Stiefel manifold, i.e., 
\begin{align}
    \|X_{k+1}^TX_{k+1}-I\|_{\F}& \leq \|(X_k+\alpha \widehat{d}_k )^T(X_k+\alpha_k \widehat{d}_k)-I\|^2_{\F} \leq (\|X_k^TX_k-I\|_{\F}+\alpha_k^2\|\widehat{g}_k\|_{\F}^2)^2\nonumber\\
    &\leq(\|X_k^TX_k-I\|_{\F}+\frac{t_k^2\|\widehat{g}_k\|_{\F}^2}{4\|\widehat{g}_k\|_{X_k}})^2\leq  0.5.\nonumber
\end{align}
The second part of the sequence satisfies Assumption~\ref{assum1} after a finite number of steps. Therefore, this strategy is equivalent to: the first few steps are used as the initialization phase, and after a finite number of steps, it conforms to the assumption of the theoretical analysis. This choice provides a practical mechanism for reconciling the numerical implementation with the theoretical conditions on ${t_k}$.
 \end{remark}
 
 \subsection{Discretized Kohn–Sham total
energy minimization and principal component analysis}\label{subsec:5.2}

\textbf{Efficiency and effectiveness:} Table~\ref{tb:pca} reports an average result of 10 random runs. We can see that IRGD-StieONS and IRSGD-StieONS are the most efficient algorithms in the sense that they find similar accurate solutions using the least computational time for both DKS and PCA. Such performance is further verified by Figure~\ref{fig:DKS-PCA}, where the result of one typical random run is shown. It can be seen from Figure~\ref{fig:DKS-PCA} that the proposed algorithms IRGD-StieONS and IRSGD-StieONS outperform the existing infeasible algorithms in terms of both absolute error and feasibility. Though the feasible method RGD maintains persistently low feasibility, the absolute error does not decrease efficiently due to difficulties in parallelization within the retraction. \\
\textbf{Parameters sensitivity:} We demonstrate the sensitivity of IRGD-StieONS to the parameter $\{\gamma_k\}$ and Landing, ExPen, PLAM to the penalty parameter $\beta$.
It can be observed from Figures~\ref{fig:DKS-beta} and~\ref{fig:PCA-beta} that the performance of IRGD-StieONS is insensitive to the choice of the parameter $\gamma_k$. 
In contrast, the other methods are more sensitive to the parameter $\beta$, where inappropriate selections may degrade convergence performance or even prevent the algorithms from converging. Specifically, for Landing, both smaller or larger values of $\beta$ slow down the decay. For ExPen, divergence occurs for both test problems when $\beta$ is set to 0.01 and 0.1. Moreover, when $\beta = 1$, ExPen-GD also diverges in DKS. For PLAM, divergence occurs on the DKS problem for $\beta$ equal to 0.01, 0.1, and 1.
\begin{table}[htbp]
\centering
\small
\caption{An average of 10 random runs with 10 random seeds. For deterministic algorithms, one epoch corresponds to a single iteration, whereas for stochastic algorithms, one epoch corresponds to one complete pass through the dataset. Computational times are in seconds. }\label{tb:pca}
\setlength{\tabcolsep}{2pt}
\begin{tabular}{cccccc}
\hline
\textbf{Problem} & \textbf{Method} & \textbf{Absolute error} & \textbf{Feasibility} & \textbf{Epoch} & \textbf{Time} \\
\hline
\multirow{4}{*}{\makecell{DKS \\
$(n, p) = (1000, 50)$}} &  Landing & 6.27e-13 & 3.12e-15 & 384 & 3.00 \\
& RGD & 9.09e-13 & 2.31e-15 & 520 & 5.16 \\
& ExPen & 5.14e-13 & 1.57e-15 & 524 & 4.07 \\
& PLAM & 5.68e-13 & 6.74e-15 & 663 & 11.35 \\
& IRGD-StieONS & \textbf{4.96e-13} & \textbf{1.25e-15} & \textbf{366} & \textbf{2.90} \\
\hline
\multirow{4}{*}{\makecell{PCA\\$(n,m,p)=(1000,500,100)$ }} &  Landing & \textbf{5.68e-14} & 7.84e-15 & 412 & 1.73 \\
& RGD & 9.66e-13 & 4.06e-15 & 510 & 2.59 \\
& ExPen & \textbf{5.68e-14} & 5.25e-15 & 1122 & 4.57 \\
& PLAM & \textbf{5.68e-14} & 9.51e-15 & 696 & 3.68 \\
& IRGD-StieONS & 9.95e-14 & \textbf{2.05e-15} & \textbf{272} & \textbf{1.13} \\
\hline
\multirow{4}{*}{\makecell{PCA\\$(n,m,p) = (10000,1000,100)$\\ batch size = 500}} 
& Landing-SG & \textbf{5.27e-13} & 9.03e-15 & 324 & 37.80 \\
& RSGD & 9.63e-13 & 3.33e-15 & 1339 & 123.36 \\
& ExPen-SGD & 9.66e-13 & 1.83e-15 & 1481 & 155.38 \\
& PLAM-SG & 9.66e-13 & 3.52e-15 & 598 & 67.87 \\
& IRSGD-StieONS & 5.66e-13 & \textbf{1.80e-15} & \textbf{241} & \textbf{32.84} \\

\hline
\end{tabular}
\end{table}
\begin{figure}[htbp]  
    \centering
\includegraphics[width=0.98\textwidth]{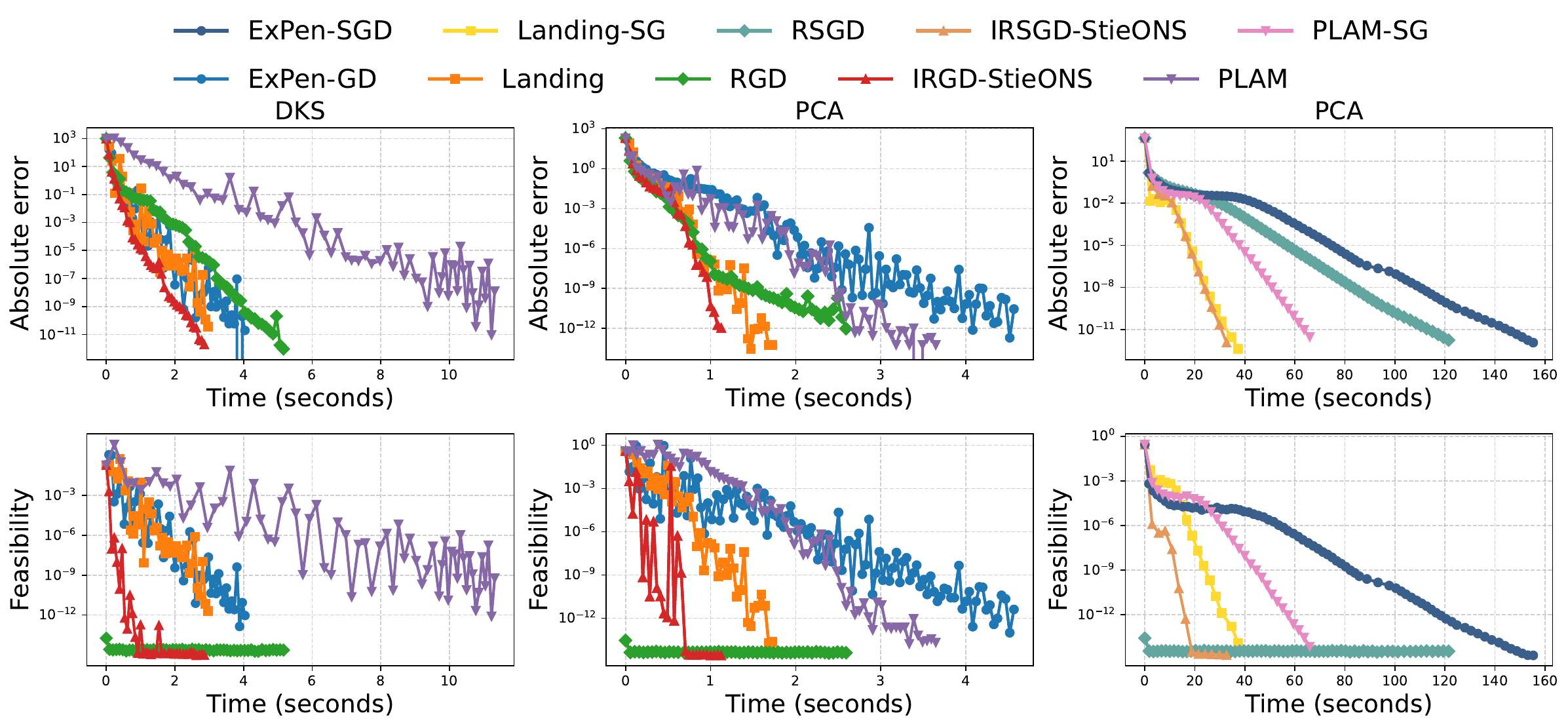}
    \caption{ The first column reports the experimental results in DKS. The second and third columns report the experimental results in PCA with $(n, m, p) = (1000, 500, 100)$ and $(n, m, p, \mathrm{batch\; size}) = (10000, 1000, 1000, 500)$, respectively. The curves are plotted by sampling points at fixed iteration‑step intervals for better visualization.
    }
    \label{fig:DKS-PCA}
\end{figure}
\begin{figure}[htbp]  
    \centering
\includegraphics[width=0.97\textwidth]{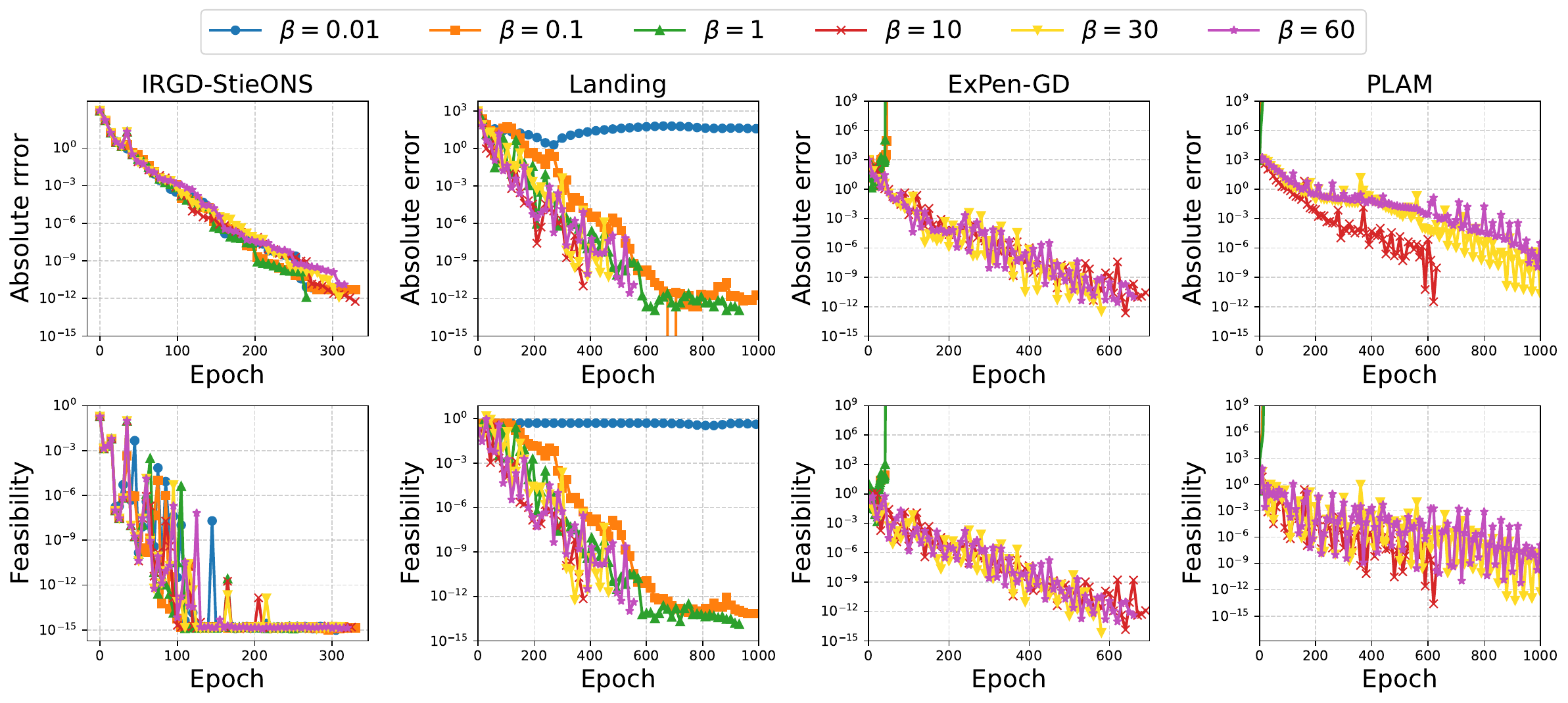}
\caption{Sensitivity to the parameter $\beta$ in DKS. The curves are plotted by sampling points at fixed iteration‑step intervals for better visualization.}
\label{fig:DKS-beta}
\end{figure}

\begin{figure}[htbp]  
    \centering
\includegraphics[width=0.97\textwidth]{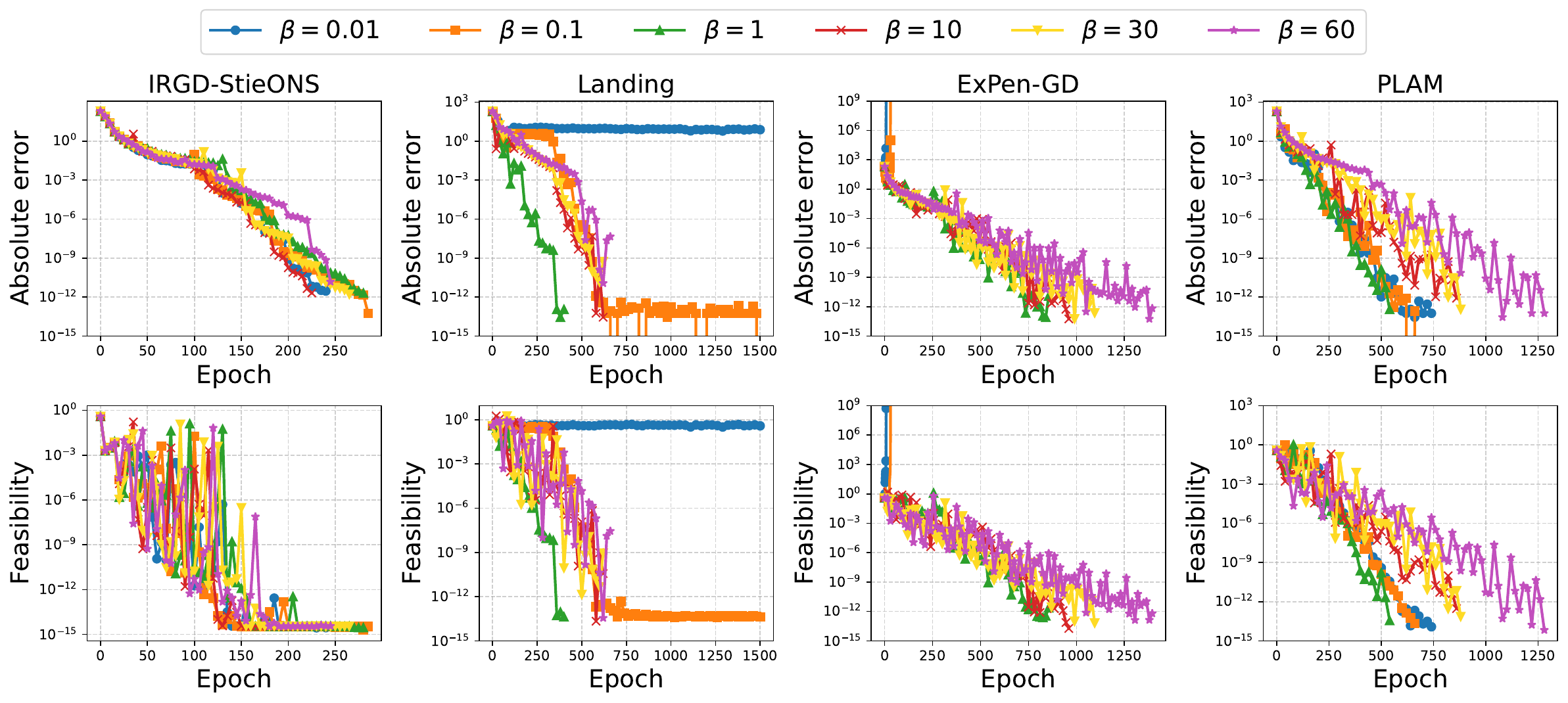}
\caption{Sensitivity to the parameter $\beta$ in PCA. The curves are plotted by sampling points at fixed iteration‑step intervals for better visualization.}
\label{fig:PCA-beta}
\end{figure}

\subsection{CNN with Orthogonality Constraints}\label{subsec:5.3}
To evaluate the robustness of the proposed method against input perturbations, Gaussian white noise is added to the training data. Specifically, the noise standard deviations are selected as std$ = 0.05, 0.1, 0.15$, where std = 0.05 is a slight noise, std = 0.1 is a moderate noise, and std = 0.15 is a strong noise. Figure~\ref{fig:vgg-std0} shows the feasibility and the test accuracy against the CPU time when the standard deviation of the noise is from 0 to 0.15.
For std=0, IRGD-StieONS and Landing-SG achieve comparable test accuracy and convergence speed. Nevertheless, IRGD-StieONS maintains a smaller constraint violation of the convolutional kernels, indicating better preservation of the Stiefel manifold structure during training. This improved feasibility contributes to more stable feature representations and enhanced robustness under perturbed inputs. The superiority of IRGD-StieONS is more evident under moderate noise level, whereas the performance differences become less significant under strong noise level due to the substantial loss of discriminative information.

In addition, Landing-SG, Expen-SGD, and PLAM-SG are sensitive to the choice of penalty parameter. Figure~\ref{fig:vgg-gammas} illustrates the results after tuning the penalty parameter $\beta$ with a noise std = 0.1. After parameter tuning, its performance is improved; however, the proposed method still achieves better performance.

 \begin{figure}[htbp]  
    \centering
\includegraphics[width=0.97\textwidth]{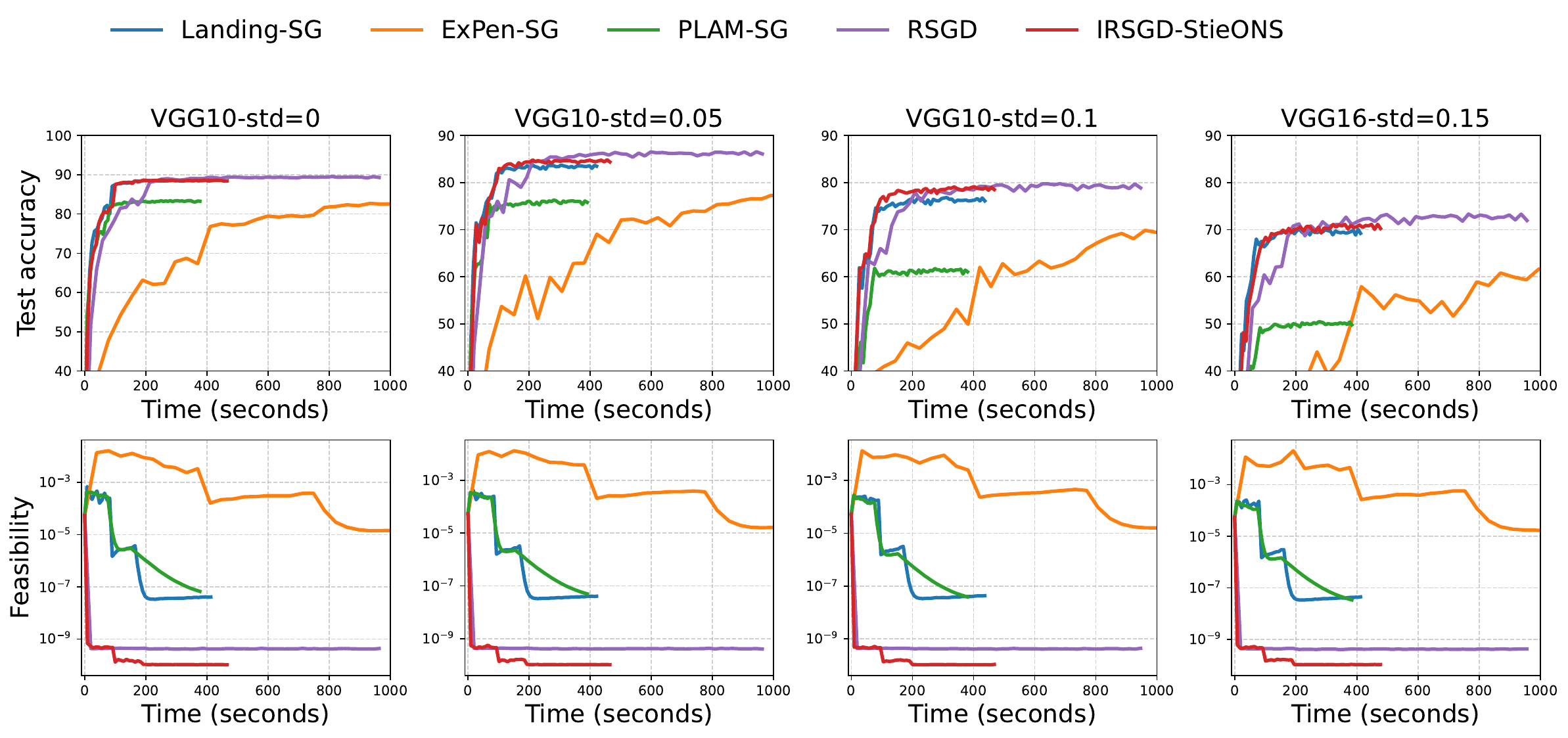}
    \caption{ VGG16 experiments.}
    \label{fig:vgg-std0}
\end{figure}

 \begin{figure}[htbp]  
    \centering
\includegraphics[width=0.97\textwidth]{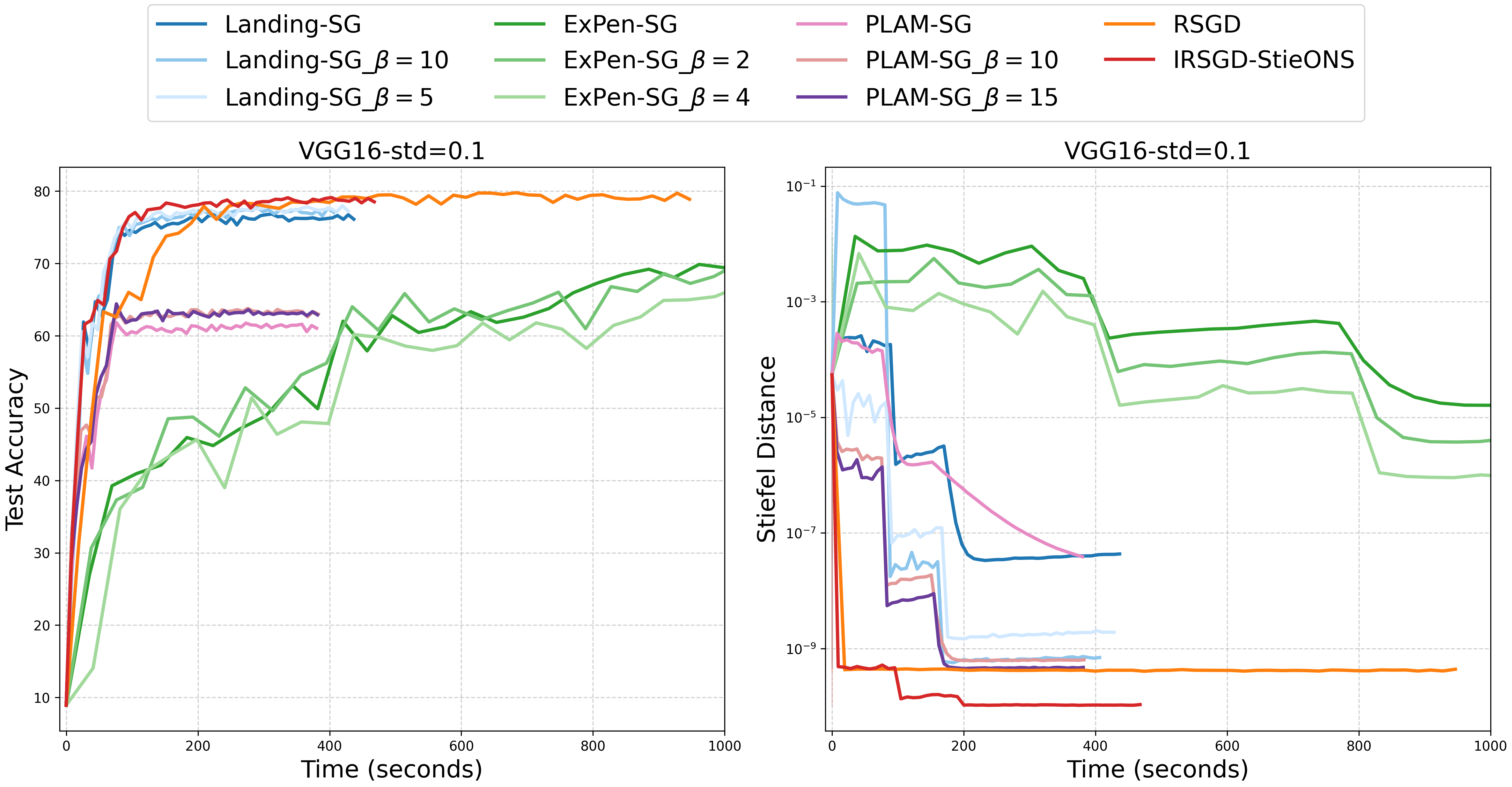}
    \caption{ VGG16 experiments with different parameter $\beta.$}
    \label{fig:vgg-gammas}
\end{figure}
 \section{Conclusion and future work}\label{Sec6}
In this paper, we proposed an inexact Riemannian gradient descent algorithm on Stiefel manifold with One Newton-Schulz iteration (IRGD-StieONS) for optimization problems on the Stiefel manifold. Unlike existing approaches that rely on retraction or penalty parameter to handle the manifold constraint, the proposed method employs a one step Newton-Schulz iteration to gradually restore feasibility, avoiding the introduction of additional penalty parameters.
Furthermore, we extended IRGD-StieONS to stochastic optimization settings and developed a stochastic variant for large-scale finite-sum problems. The convergence properties of both deterministic and stochastic algorithms were established under appropriate assumptions.

Extensive numerical experiments on DKS, PCA and orthogonally constrained convolutional neural networks demonstrate the effectiveness and robustness of the proposed methods. The results show that IRGD-StieONS and its stochastic variant achieve competitive performance while maintaining better feasibility of the Stiefel constraint and reducing sensitivity to parameter selection.

Future work will focus on exploring more efficient variants for large-scale problems, such as incorporating variance reduction techniques and extending the proposed framework to other structured optimization problems.
\section{Appendix}\label{Appendix}
The following properties in Newton-Schulz iteration are from the existing literature, will be used in the proofs of Lemma~\ref{B117} and Lemma~\ref{lemmaxx}. We restate it here without proof for completeness.

    Let $Z\in\mathbb{R}^{n\times p}$ and $\|I-Z^TZ\|_{\F}<1$. It follows from~\cite{n1990polar} that the Newton–Schulz iteration converges and satisfies 
    \begin{align}
        \|I-\Psi(Z)^T\Psi(Z)\|_{\F}<\|I-Z^TZ\|_{\F}^2.\label{ZTZ}
    \end{align}
    Moreover, by~\cite[Lemma~2.4]{CXJ2025TightError},
\begin{equation}\label{pp}
    \|A-\P_{\St(p,n)}(A)\|_{\F}\leq \|I-A^TA\|_{\F},\forall A\in\mathbb{R}^{n\times p}.
\end{equation}
Applying \eqref{pp} to $A = \Psi(Z)$ gives
\begin{equation}\label{p1}
\|\Psi(Z) - \Psi^{(\infty)}(Z)\|_{\F} \leq \|I - \Psi(Z)^T \Psi(Z)\|_{\F}.
\end{equation}

\subsection{Proof of Lemma~\ref{B117}}\label{appendix1}
\begin{proof}
By definition $ \theta_0=0.5$. 
Suppose at $k$, it holds that $\theta_k\leq0.5$. Then $\theta_{k+1} = (\theta_k+t_k^2)^2\leq (0.5+0.2)^2\leq 0.5. $ By the principle of mathematical induction, $\theta_k\leq 0.5$ for all $k.$
Next, we show that \( \theta_k \to 0 \). 
Let \( \liminf_{k \to \infty} \theta_k = m \).  Taking a subsequence \( \theta_{k_j} \to m \) and using
$
\theta_{k_j+1} = (\theta_{k_j} + t_{k_j}^2)^2,
$
together with \( t_k \to 0 \), we obtain \( \theta_{k_j+1} \to m^2 \). Hence \( m \le m^2 \), which implies \( m = 0 \) or \( m \ge 1 \). Since \( \theta_k \le 0.5 \), it follows that \( m = 0 \).
Similarly, letting \( \limsup_{k \to \infty} \theta_k = M \), one can show that \( M \ge \sqrt{M} \), which yields \( M = 0 \). Therefore, \( \theta_k \to 0 \).

Define $r_k := \theta_k / t_k^4$, then
$
r_k = \frac{t_{k-1}^4}{t_k^4} \big(r_{k-1} t_{k-1}^2 + 1\big)^2.
$
We next show that $\{r_k\}$ is bounded. Otherwise, 
we can
 define a subsequence $\{r_{k_j}\}$ by $k_1=1$ and $k_{j}=\min\{i>k_{j-1} :  r_i > r_{k_{j-1}}\}$ for $j>1$. 
It follows that $r_{k_j} \geq r_i$ for all $i < k_j$ and $r_{k_j}\to\infty$.
From the recursion, we have
$
r_{k_j} =
\frac{t_{k_j-1}^4}{t_{k_j}^4}
\big(r_{k_j-1} t_{k_j-1}^2 + 1\big)^2.
$
Dividing both sides by $r_{k_j}$ yields
\begin{equation} \label{eq02}
1 =
\frac{t_{k_j-1}^4}{t_{k_j}^4}
\left(
\frac{r_{k_j-1} t_{k_j-1}^2 + 1}{\sqrt{r_{k_j}}}
\right)^2.
\end{equation}
Since $r_{k_j} \ge r_{k_j-1}$, we have $
\frac{r_{k_j-1} t_{k_j-1}^2}{\sqrt{r_{k_j}}}
\le \sqrt{r_{k_j-1}}\, t_{k_j-1}^2
= \sqrt{\theta_{k_j-1}} \to 0.
$
Moreover, $1/\sqrt{r_{k_j}} \to 0$. Hence
$
\frac{r_{k_j-1} t_{k_j-1}^2 + 1}{\sqrt{r_{k_j}}} \to 0.
$
Using $t_{k_j-1}/t_{k_j} \to 1$ for~\eqref{eq02} yields $0 = 1$, which is a contradiction. Therefore, $\{r_k\}$ is bounded. Letting $k \to \infty$ in $r_k = \frac{t_{k-1}^4}{t_k^4} \big(r_{k-1} t_{k-1}^2 + 1\big)^2$ yields $\lim_{k\to \infty} r_k = 1$, i.e., $\theta_k = O(t_k^4)$. It follows from $\sum t_k^4 < \infty$ that $\sum_{k = 0}^{\infty}\theta_k<\infty.$
     
    The first inequality of~\eqref{r7} follows from~\eqref{pp}. The second inequality of~\eqref{r7} can be proved by induction. First, it holds for $k = 0$ by definition. Suppose the second inequality of~\eqref{r7} holds for $k$. Then the eigenvalues of $I_n-\frac{1}{2}X^T_kX_k$ are between $\frac{1}{4}$ and $\frac{3}{4}$. It follows that $\|Y\|_{X_k}\geq \frac{1}{2}\|Y\|_{\F}$ for any $Y\in\mathbb{R}^{n\times p}.$ Therefore,
    \begin{align}
         \|I-X_{k+1}^T X_{k+1}\|_{\F} 
            &\leq \|I-(X_{k}+\alpha_{k} \widehat{d}_{k})^{T}(X_{k}+\alpha_{k}\widehat{d}_{k})\|_{\F}^2 \leq (\theta_{k} +t_{k}^2\frac{\|\widehat{g}_k\|_{\F}^2}{4\|\widehat{g}_k\|_{X_k}^2})^2\nonumber\\
            &\leq (\theta_k + t_k^2)^2 =  \theta_{k+1},
    \end{align}
where the first inequality follows from~\eqref{ZTZ} and $\|I-(X_{k}+\alpha_{k}\widehat{d}_{k})^{T}(X_{k}+\alpha_{k}\widehat{d}_{k})\|_{\F}  \leq \theta_{k} +t_{k}^2\frac{\|\widehat{g}_k\|_{\F}^2}{4\|\widehat{g}_k\|_{X_k}^2} \leq 0.7$ and the second inequality follows from the triangle inequality and $2\alpha_k\|\widehat{g}_k\|_{X_k}\leq t_k$. Therefore, the second inequality of~\eqref{r7} holds for $k+1$, which completes the proof.

\end{proof}
\subsection{Proof of Lemma~\ref{lemmaxx}} \label{proofoflemmaxx}
\begin{proof}
By~\eqref{p1}, 
     \begin{align}\label{b1}
         \|X_k-\dot{X}_k\|_{\F}\leq \|I-X_k^TX_k\|_{\F}= \delta_k, \text{ for all }k. 
     \end{align}
     It follows from (\ref{L1}) and~\eqref{b1} that 
\begin{align}
    \|\hat{g}_k-\dot{g}_k\|_{\F}&\leq L_g \|X_k-\dot{X}_k\|_{\F}\leq  L_g\delta_k, \text{ for all }k. \label{gg}
\end{align}
By the triangle inequality, we have
    \begin{equation*}
    \begin{aligned}
        \|\Psi({X_k}+\frac{\alpha_k}{\lambda}\hat{d}_k)-R_{\dot{X}_k}(\frac{\alpha_k}{\lambda}\dot{d}_k)\|_{\F}&\leq \underbrace{\|\Psi({X_k}+\frac{\alpha_k}{\lambda}\hat{d}_k)-\P_{\St(p,n)}(X_k+\frac{\alpha_k}{\lambda}\widehat{d}_k)\|_{\F}}_{(I)}
        \\
        &+\underbrace{\|\P_{\St(p,n)}(X_k+\frac{\alpha_k}{\lambda}\widehat{d}_k)- R_{\dot{X}_k}(\frac{\alpha_k}{\lambda}\dot{d}_k)\|_{\F}}_{(II)}.\\
    \end{aligned}
    \end{equation*}
   For term (I), since $\|I-({X_k}+\frac{\alpha_k}{\lambda}\hat{d}_k)^T({X_k}+\frac{\alpha_k}{\lambda}\hat{d}_k)\|_{\F}\leq \delta_k+\frac{t_k^2}{\lambda^2} \leq 0.7,$ we have 
\begin{align}
    \|\Psi({X_k}+\frac{\alpha_k}{\lambda}\hat{d}_k)-\P_{\St(p,n)}(X_k+\frac{\alpha_k}{\lambda}\widehat{d}_k)\|_{\F}&\leq \|I-(\Psi({X_k}+\frac{\alpha_k}{\lambda}\hat{d}_k))^T\Psi({X_k}+\frac{\alpha_k}{\lambda}\hat{d}_k)\|_{\F}\text{ (by~\eqref{p1})}\nonumber\\
    &\leq \|I-({X_k}+\frac{\alpha_k}{\lambda}\hat{d}_k)^T({X_k}+\frac{\alpha_k}{\lambda}\hat{d}_k)\|_{\F}^2\text{ (by~\eqref{ZTZ})} \nonumber\\
    &\leq (\delta_k+\frac{\alpha_k^2\|\widehat{g}_k\|_{\F}^2}{\lambda^2} )^2\leq\frac{1}{\lambda^4}\min\{\theta_{k+1},(\delta_k+\Bar{\alpha}^2\|\widehat{g}_k\|_{\F}^2 )^2\},\label{r2}
    \end{align}
   where 
   the third inequality follows from the triangle inequality, $\delta_k = \|I-X_k^TX_k\|_{\F}$, the fourth inequality follows from $\lambda < 1$, $(\delta_k+\alpha_k^2\|\widehat{g}_k\|_{\F}^2 )^2\leq (\theta_k + t_k^2)^2 = \theta_{k+1} $, and $(\delta_k+\alpha_k^2\|\widehat{g}_k\|_{\F}^2 )^2\leq (\delta_k+\Bar{\alpha}^2\|\widehat{g}_k\|_{\F}^2 )^2.$

   For term (II), since $R$ is a polar‑based retraction, we have $R_{\dot{X}_k}(\frac{\alpha_k}{\lambda}\dot{d}_k) = \P_{\St(p,n)}({\dot{X}_k}+\frac{\alpha_k}{\lambda}\dot{d}_k) $. Consequently,  $\|\P_{\St(p,n)}(X_k+\frac{\alpha_k}{\lambda}\widehat{d}_k)- R_{\dot{X}_k}(\frac{\alpha_k}{\lambda}\dot{d}_k)\|_{\F}=\|\P_{\St(p,n)}(X_k+\frac{\alpha_k}{\lambda}\widehat{d}_k)- \P_{\St(p,n)}({\dot{X}_k}+\frac{\alpha_k}{\lambda}\dot{d}_k)\|_{\F}$. Since $X_k+\frac{\alpha_k}{\lambda}\widehat{d}_k, \dot{X}_k+\frac{\alpha_k}{\lambda}\dot{d}_k\in\Omega$, we have
        \begin{align}
         \|\P_{\St(n,p)}(X_k+\frac{\alpha_k}{\lambda}\widehat{d}_k)- \P_{\St(n,p)}({\dot{X}_k}+\frac{\alpha_k}{\lambda}\dot{d}_k)\|_{\F} 
           &\leq L_p(\|X_k-\dot{X}_k\|_{\F}+\frac{\alpha_k}{\lambda}\|\hat{d}_k-\dot{d}_k\|_{\F})\nonumber\\
           &\leq L_p\delta_k + \frac{\Bar{\alpha}L_pL_g}{\lambda}\delta_k,\label{r3}
        \end{align}
   where the second inequality follows from~\eqref{b1},~\eqref{gg} and $\alpha_k\leq \Bar{\alpha}$. Combing~\eqref{r2} and~\eqref{r3}, we can obtain
    \begin{equation*}\label{123}
        \begin{aligned}
        \|\Psi({X_k}+\frac{\alpha_k}{\lambda}\hat{d}_k)-R_{\dot{X}_k}(\frac{\alpha_k}{\lambda}\dot{d}_k)\|_{\F}&\leq \frac{1}{\lambda^4}\min\{\theta_{k+1},(\delta_k+\Bar{\alpha}^2\|\widehat{g}_k\|_{\F}^2 )^2\} + L_p\delta_k + \frac{\Bar{\alpha}L_pL_g}{\lambda}\delta_k \leq \frac{1}{\lambda^4}\zeta_{k+1},
        \end{aligned}
    \end{equation*}
     where the second inequality follows from a proper scaling. 
     
     For the conclusion $\|\Psi({X_k}+\alpha_k\hat{d}_k)-R_{\dot{X}_k}(\alpha_k\dot{d}_k)\|_{\F}\leq \zeta_{k+1}$, the proof is similar to the above procedure and is omitted here.

    \end{proof}

\printbibliography
\end{document}